\UseRawInputEncoding
\documentclass[12pt]{amsart}
\usepackage{}

\usepackage{cancel}
\usepackage{amsmath}
\usepackage{amsfonts}
\usepackage{amssymb}
\usepackage[all]{xy}           

\usepackage{bbding}
\usepackage{txfonts}
\usepackage{amscd}

\usepackage[shortlabels]{enumitem}
\usepackage{ifpdf}
\ifpdf
  \usepackage[colorlinks,final,backref=page,hyperindex]{hyperref}
\else
  \usepackage[colorlinks,final,backref=page,hyperindex,hypertex]{hyperref}
\fi
\usepackage{tikz}
\usepackage[active]{srcltx}

\makeatletter

\newtheorem{thm}{Theorem}[section]

\newtheorem{cor}[thm]{Corollary}
\newtheorem{pro}[thm]{Proposition}
\newtheorem{ex}[thm]{Example}
\newtheorem{rmk}[thm]{Remark}
\newtheorem{defi}[thm]{Definition}

\newcommand {\emptycomment}[1]{}

\newcommand{\lon }{\,\rightarrow\,}
\newcommand{\be }{\begin{equation}}
\newcommand{\ee }{\end{equation}}

\newcommand{\g}{\mathfrak g}
\newcommand{\h}{\mathfrak h}

\newcommand{\huaB}{\mathcal{B}}

\newcommand{\aaar}{\substack{\longrightarrow\\[-0.85em] \longrightarrow \\[-0.85em] \longrightarrow}}

\newcommand{\huaI}{\mathcal{I}}

\newcommand{\huaO}{{\mathcal{O}}}

\newcommand{\dM}{\mathrm{d}}

\newcommand{\Courant}[1]{\left\llbracket  #1\right\rrbracket }

\newcommand{\Id}{{\rm{Id}}}

\newcommand{\br}[1]{   [ \cdot,    \cdot  ]   }

\newcommand{\Der}{\mathrm{Der}}

\newcommand{\Ad}{\mathrm{Ad}}
\newcommand{\Aut}{\mathrm{Aut}}

\newcommand{\gl}{\mathfrak {gl}}

\newcommand{\so}{\mathfrak {so}}

\newcommand{\ad}{\mathrm{ad}}

\begin{document}

\title{On the integration of relative Rota-Baxter Lie algebras}

\author{Jun Jiang}
\address{Department of Mathematics, Jilin University, Changchun 130012, Jilin, China}
\email{junjiang@jlu.edu.cn}

\author{Yunhe Sheng}
\address{Department of Mathematics, Jilin University, Changchun 130012, Jilin, China}
\email{shengyh@jlu.edu.cn}

\author{Chenchang Zhu}
\address{Mathematics Institute, Georg-August-University Gottingen, Bunsenstrasse 3-5 37073, Gottingen, Germany}
\email{czhu@gwdg.de}


\begin{abstract}
In this paper, we give the necessary and sufficient conditions of the integrability of relative Rota-Baxter Lie algebras via double Lie groups, matched pairs of Lie groups and factorization of diffeomorphisms respectively. We use the integrability of Rota-Baxter operators to characterize whether the Poisson-Lie group integrating a factorizable Lie bialgebra is again factorizable. We thoroughly study the integrability of Rota-Baxter operators on the unique nontrivial 2-dimensional Lie algebra.  As a byproduct,  we construct a matched pair of Lie algebras that can not be integrated to a matched pair of Lie groups.
\end{abstract}


\renewcommand{\thefootnote}{}
\footnotetext{Corresponding Author: Yunhe Sheng}

\maketitle

\tableofcontents

\allowdisplaybreaks


\section{Introduction}
The notion of Rota-Baxter operators on associative algebras was introduced by G. Baxter in his study of fluctuation theory in probability \cite{Bax}. They are applied in the Connes-Kreimer's algebraic approach to renormalization of quantum field theory~\cite{CK,Gub}. In the corresponding semi-classical world, the notion of relative Rota-Baxter operators (also called $\huaO$-operators) on Lie algebras was introduced in \cite{Ku}, and they are closely related to various classical Yang-Baxter equations. A Rota-Baxter operator of weight $0$ on a Lie algebra is naturally the operator form of a classical skew-symmetric $r$-matrix \cite{STS}. Relative Rota-Baxter operators of weight $0$ on Lie algebras give rise to solutions of the classical Yang-Baxter equation in the semidirect product Lie algebras \cite{Bai}. Rota-Baxter operators of weight $1$ are in one-to-one
correspondence with solutions of the modified Yang-Baxter equation, and have close relation with factorizable Lie bialgebras \cite{LS}. In addition, a Lie group with zero Schouten curvature gives rise a Rota-Baxter operator of weight $1$ on a Lie algebra \cite{Kos}. In the sequel, for simplicity, we will call relative Rota-Baxter operators of weight $1$ on Lie algebras relative Rota-Baxter operators on Lie algebras when there is no ambiguity.

The notion of Rota-Baxter operators on  Lie groups  was introduced in \cite{GLS}, then this concept was generalized to relative Rota-Baxter operators on Lie groups in \cite{JSZ}. Rota-Baxter operators on Lie groups give rise to factorizations of the Lie groups, and can be applied to study integrable systems \cite{GLS, STS}. Rota-Baxter operators on Lie groups can also be used to construct set-theoretical solutions of the Yang-Baxter equation \cite{BaG, BaG2}. It was shown in \cite{GLS,JSZ} that by differentiating   (relative) Rota-Baxter operators on Lie groups, one can obtain   (relative) Rota-Baxter operators on the corresponding Lie algebras.

The converse direction, namely whether relative Rota-Baxter operators on Lie algebras can be integrated to relative Rota-Baxter operators on Lie groups is still not clear. The process from (relative) Rota-Baxter Lie algebras to (relative) Rota-Baxter Lie groups is called the integration. The first attempt in this direction is give in \cite{JSZ}, where it was shown that (relative) Rota-Baxter operators on Lie algebras can be integrated to {\em local} (relative) Rota-Baxter operators on Lie groups.  In general, the global integrability of an algebraic structure  has some topological obstruction. For example, the obstruction of integrability for  infinite-dimensional Lie algebras and that for Lie algebroids are studied in \cite{ neeb:cent-ext-gp, wz:int, cf} respectively. Very recently, some new results have appeared on the integrability of Rota-Baxter operators: \cite{Rat} shows that every Rota-Baxter operator on the Heisenberg Lie algebra can be integrated into a Rota-Baxter operator on the Heisenberg Lie group; while \cite{GKST} studies a version of formal integration of complete Rota-Baxter Lie algebras.

In this paper, we study the global integrability of relative Rota-Baxter operators on Lie algebras using different approaches. We give the necessary and sufficient integrability conditions of  relative Rota-Baxter operators via double Lie groups, matched pairs of Lie groups and factorization of diffeomorphisms respectively. As applications, we use the integrability of Rota-Baxter operators to characterize when exactly a Poisson-Lie group of a factorizable Lie bialgebra is factorizable again.  We also construct some examples that Rota-Baxter Lie algebras can not be integrated to Rota-Baxter Lie groups. As a byproduct,  we construct a matched pair of Lie algebras that can not be integrated to a matched pair of Lie groups.

\emptycomment{
{\bf Main Theorem $1$.} Let $((\g, \h), \phi, B)$ be a relative Rota-Baxter Lie algebra. The map $B:\h\lon\g$ can be integrated to a relative Rota-Baxter operator of Lie groups if and only if
\begin{itemize}
\item[\rm(i)] $D$ is a closed subgroup;
\item[\rm(ii)] $D\cap G=(e_G, e_H)$ and $G\ltimes_\Phi H=D\cdot_\Phi G$.
\end{itemize}

{\bf Main Theorem $2$.} Let $((\g, \h), \phi, B)$ be a relative Rota-Baxter Lie algebra. The map $B:\h\lon\g$ can be integrated into a relative Rota-Baxter operator of Lie groups if and only if the matched pair $(\g, Gr(B); \bar{\phi}, \bar{\theta})$ can be integrated into a matched pair of Lie groups $(G, E; \bar{\Phi}, \bar{\Theta})$.

{\bf Main Theorem $3$.} Let $((\g, \g), \ad, B)$ be a Rota-Baxter Lie algebra. The map $B:\g\lon\g$ is integrable if and only if the map $J: D\lon G$ is invertible.
}
The paper is organized as follows. In Section $2$, we use three approaches to give the necessary and sufficient conditions of   relative Rota-Baxter Lie algebras being integrable. In Section $3$, we use the integrability of Rota-Baxter operators on Lie algebras to characterize when Poisson-Lie groups integrating factorizable Lie bialgebras are factorizable Poisson-Lie groups. In Section $4$, we study the integrability of Rota-Baxter operators on the nontrivial $2$-dimensional Lie algebra. In particular, we give a concrete example of a matched pair of Lie algebras that can not be integrated to a matched pair  of Lie groups.



\vspace{2mm}
\noindent
{\bf Acknowledgements.} This research is supported by NSFC(12471060,12401076), China Postdoctoral Science Foundation (2023M741349) and DFG ZH 274/5-1.

\emptycomment{
\begin{ex}
Consider a $2$-dimensional real Lie algebra
$$\Big(\g=\text{span}\{e_1=\left(
\begin{array}{cc}
1 & 0 \\
0 & -1 \\
\end{array}
\right), e_2=\left(
\begin{array}{cc}
0 & 1 \\
0 & 0 \\
\end{array}
\right)
\}, ~~~~[\cdot, \cdot]\Big),$$ where $[e_i, e_j]=e_ie_j-e_je_i, i, j\in\{1, 2\}$.
Define a linear map $B:\g\lon\g$ by
$$
B(e_1, e_2)=(e_1, e_2)
\left(
\begin{array}{cc}
b_{11} & b_{12} \\
b_{21} & b_{22} \\
\end{array}
\right).
$$
Then $B$ is a Rota-Baxter operator if and only if
$$
[B(e_1), B(e_2)]=B([B(e_1), e_2]+[e_1, B(e_2)]+[e_1, e_2]),
$$
which implies that
$$
B=
\left(
\begin{array}{cc}
\lambda & 0 \\
\mu & 0 \\
\end{array}
\right), \quad \text{or} \quad
B=
\left(
\begin{array}{cc}
\lambda & 0 \\
\mu & -1 \\
\end{array}
\right)
\quad \text{or} \quad B=
\left(
\begin{array}{cc}
-m-1 & k \\
\frac{-m^2-m}{k} & m \\
\end{array}
\right), \quad k\neq0.
$$
\end{ex}
}

\section{Integration of relative Rota-Baxter operators}
In \cite{JSZ}, it was shown that any relative Rota-Baxter Lie algebra can be integrated into a relative Rota-Baxter Lie group locally. In this section we give the necessary and sufficient conditions of relative Rota-Baxter Lie algebras being global integrable.
We first recall some  basic results of relative Rota-Baxter Lie algebras and relative Rota-Baxter Lie groups.

\begin{defi}\rm(\cite{BGN, BD})
Let $(\g, [\cdot, \cdot]_\g)$ and $(\h, [\cdot, \cdot]_\h)$ be Lie algebras, $\phi:\g\lon\Der(\h)$ be a Lie algebra homomorphism. A linear map $B:\h\lon\g$ is called a relative Rota-Baxter operator of weight $1$ if
\begin{equation*}
[B(u), B(v)]_\g=B(\phi(B(u))v-\phi(B(v))u+[u, v]_\h), \quad \forall u, v\in\h.
\end{equation*}
The quadruple $(\g, \h, \phi, B)$ is called a relative Rota-Baxter Lie algebra of weight $1$. In particular, if $\h=\g$ and the action is the adjoint representation $\ad$ of $\g$ on itself, then $B$ is called a Rota-Baxter operator of weight $1$, and $(\g, B)$ is called a Rota-Baxter Lie algebra of weight $1$.
\end{defi}

\begin{rmk}
In the sequel, we will call relative Rota-Baxter Lie algebras of weight $1$ simply relative Rota-Baxter Lie algebras when there is no ambiguity.
\end{rmk}

\begin{defi}\rm(\cite{GLS, JSZ})
Let $(G, \cdot_G)$ and $(H, \cdot_H)$ be Lie groups, $\Phi:G\lon\Aut(H)$ be an action of $G$ on $H$ . A smooth map $\huaB: H\lon G$ is called a relative Rota-Baxter operator if
\begin{equation*}
\huaB(a)\cdot_G\huaB(b)=\huaB(a\cdot_H\Phi(\huaB(a))b), \quad\forall a, b\in H.
\end{equation*}
The quadruple $(G, H, \Phi, \huaB)$ is called a relative Rota-Baxter Lie group. In particular, if $H=G$ and the action is the adjoint representation $\Ad$ of $G$ on itself, then $\huaB$ is called a Rota-Baxter operator, and $(G, \huaB)$ is called a Rota-Baxter Lie group.
\end{defi}

There is a differentiation relation between relative Rota-Baxter Lie groups and relative Rota-Baxter Lie algebras as following.

\begin{thm}\cite{GLS, JSZ}
Let $(G, H, \Phi, \huaB)$ be a relative Rota-Baxter Lie group. Denote by $\g$ and $\h$ Lie algebras of $G$ and $H$ respectively. Denote by $\phi:\g\lon\Der(\h)$ the Lie algebra action comes from the Lie group action $\Phi:G\lon\Aut(H)$. Then $(\g, \h, \phi, \huaB_{*_{e_H}})$ is a relative Rota-Baxter Lie algebra.
\end{thm}

Let $(\g, \h, \phi, B)$ be a relative Rota-Baxter Lie algebra. Let $G$ and $H$ be the connected and simply connected Lie
groups integrating $\g$ and $\h$ respectively. Since $\phi:\g\lon\Der(\h)$ is a Lie algebra homomorphism, then $\phi$ can be integrated into a unique Lie group homomorphism $\tilde{\phi}:G\lon \Aut(\h)$. By the Lie II theorem, it shows that $\Aut(\h)\cong\Aut(H)$, thus there is a unique homomorphism $\Phi: G\lon\Aut(H)$ integrating $\tilde{\phi}$. This procedure can be well explained by the following diagram:
\begin{equation*}\label{eq:relation}
\small{ \xymatrix{G\ar@{=}[d]  \ar[rr]^{\Phi} & & \Aut(H) \\
 G \ar[rr]^{\tilde{\phi}} &  &  \Aut(\h)\ar[u]_{\text{integration}}  \\
 \g\ar[u]^{\text{integration}} \ar[rr]^{  \phi=\tilde{\phi}_{*e_G} } & &\Der(\h)\ar[u]_{\text{integration}}.}
}
\end{equation*}


\subsection{Integrability via double Lie groups}
Let $(\g, [\cdot,\cdot]_\g), (\h, [\cdot,\cdot]_\h)$ be Lie algebras and $\phi:\g\lon\Der(\h)$ be a Lie algebra homomorphism. Consider the semidirect product Lie algebra $\g\ltimes_{\phi} \h$, with the Lie bracket $[\cdot, \cdot]_{\phi}$ given by
$$
[(x, u), (y, v)]_{\phi}=([x, y]_\g, \phi(x)v-\phi(y)u+[u, v]_\h), \quad \forall x, y\in\g, u, v\in\h.
$$
There is the following proposition.
\begin{pro}\rm(\cite{JSZ})\label{graphro}
Let $\phi: \g\rightarrow\Der(\h)$ be an action of a Lie algebra $(\g, [\cdot,\cdot]_{\g})$ on a Lie algebra $ (\h, [\cdot,\cdot]_{\h})$. Then a linear map $B: \h\rightarrow\g$ is a relative Rota-Baxter operator if and only if the graph   $Gr(B)=\{(B(u), u)|u\in \h\}$ is a subalgebra of the Lie algebra $(\g\ltimes_\phi\h, [\cdot, \cdot]_{\phi})$.
\end{pro}

Let $(G, \cdot_G), (H, \cdot_H)$ be Lie groups and $\Phi: G\lon\Aut(H)$ be an action of $G$ on $H$. There is the semidirect product Lie group $(G\ltimes_\Phi H, \cdot_\Phi)$, where $\cdot_\Phi$ is given by
$$
(g_1, h_1)\cdot_\Phi(g_2, h_2)=(g_1\cdot_G g_2, h_1\cdot_H\Phi(g_1)h_2), \quad\forall g_1, g_2\in G, h_1, h_2\in H.
$$
\begin{pro}\rm(\cite{JSZ})\label{graphgroup}
Let  $\Phi: G\to\Aut(H)$ be an action of a Lie group $G$ on a Lie group $ H$. Then a smooth map $\huaB: H\rightarrow G$ is a relative Rota-Baxter operator if and only if the graph   $Gr(\huaB)=\{(\huaB(h),h)|h\in H\}$ is a Lie subgroup of the  Lie group $(G\ltimes_\Phi H, \cdot_\Phi)$.

Moreover, if $\huaB:H\lon G$ is a relative Rota-Baxter operator on Lie groups, then the Lie algebra of $Gr(\huaB)$ is $Gr(\huaB_{*_{e_H}})$.
\end{pro}

\begin{defi}\cite{LuW}
Let $G$ be a Lie group and $A, B$ be Lie subgroups of $G$. The triple $(G, A, B)$ is called a double Lie group if $A$ and $B$ are both closed Lie subgroups of $G$ such that $$A\cap B=\{e_G\},\quad AB=G.$$
\end{defi}
In the following, we will use double Lie groups to explore the integration of relative Rota-Baxter Lie algebras.

 Let $(\g, \h, \phi, B)$ be a relative Rota-Baxter Lie algebra. Denote by $G$ and $H$ connected and simply connected Lie groups integrating $\g$ and $\h$ respectively. Denote by $\Phi: G\lon\Aut(H)$ the integration of $\phi:\g\lon\Der(\h)$.
We know that the semidirect product Lie group $(G\ltimes_\Phi H, \cdot_\Phi)$ is the integration of $(\g\ltimes_\phi\h, [\cdot,\cdot]_\phi)$. By Proposition \ref{graphro},  $(Gr(B), [\cdot, \cdot]_\phi)$ is a subalgebra of the Lie algebra $(\g\ltimes_\phi\h, [\cdot,\cdot]_\phi)$. Denote by $D\subseteq G\ltimes_\Phi H$ the connected Lie subgroup integrating $(Gr(B), [\cdot, \cdot]_\phi)$.

\begin{thm}\label{proint2}
With the above notations, a relative Rota-Baxter Lie algebra $(\g, \h, \phi, B)$ can be integrated into a relative Rota-Baxter Lie group $(G, H, \Phi, \huaB)$ if and only if $(G\ltimes_\Phi H, D, G\times e_H)$ is a double Lie group.
\emptycomment{
\begin{itemize}
\item[\rm(i)] $D$ is a closed Lie subgroup of $G\ltimes_\Phi H$;
\item[\rm(ii)] $D\cap (G\times e_H)=(e_G, e_H)$ and $G\ltimes_\Phi H=D\cdot_\Phi (G\times e_H)$.
\end{itemize}}
\end{thm}
\begin{proof}
If a relative Rota-Baxter Lie algebra $(\g, \h, \phi, B)$ can be integrated into a relative Rota-Baxter Lie group $(G, H, \Phi, \huaB)$, by Proposition \ref{graphgroup}, we have that $$D=Gr(\huaB)=\{(\huaB(h), h)|h\in H\}$$ is a connected Lie subgroup integrating $Gr(B)$. Since $Gr(\huaB)$ is the graph of smooth map $\huaB$, by \cite[Proposition 5.4]{Lee}, it follows that $Gr(\huaB)$ is an embedded submanifold. By \cite[Proposition 9.3.9]{Neeb}, we have that $Gr(\huaB)$ is a connected closed Lie subgroup. Moreover, by direct calculation, it follows that
$$
Gr(\huaB)\cap (G\times e_H)=(e_G, e_H),
$$
and
$$
Gr(\huaB)\cdot_\Phi (G\times e_H)=\{(\huaB(h)\cdot_G g, h)|\forall (\huaB(h), h)\in Gr(\huaB), g\in G\}=G\ltimes_{\Phi} H.
$$
Thus $(G\ltimes_\Phi H, Gr(\huaB), G\times e_H)$ is a double Lie group.

Conversely, assume that $(G\ltimes_\Phi H, D, G\times e_H)$ is a double Lie group, it follows that $D$ is a closed Lie subgroup of $G\ltimes_\Phi H$ and
$$
D\cap (G\times e_H)=(e_G, e_H), \quad G\ltimes_\Phi H=D\cdot_\Phi (G\times e_H),
$$
it means that $D$ is an embedded submanifold of $G\ltimes_\Phi H$. Define smooth maps $P_H: G\ltimes_\Phi H\lon H$ and $P_G: G\ltimes_\Phi H\lon G$ by
$$
P_G(g, h)=g, \quad P_H(g, h)=h, \quad \forall g\in G, h\in H.
$$
Then we have that the map $P_H:D\lon H$ is a smooth map. Since $$D\cap (G\times e_H)=(e_G, e_H)$$ and
$$
G\ltimes_\Phi H=D\cdot_\Phi (G\times e_H)=\{(g'\cdot_G g, h)|\forall (g', h)\in D, g\in G\},
$$
then it follows that $P_H: D\lon H$ is subjective. In addition, if $(g_1, h)\in D$ and $(g_2, h)\in D$, since
$$
(g_2, h)\cdot_{\Phi}(g_1, h)^{-1}=(g_1^{-1}\cdot_G g_2, e_H)\in (G\times e_H)\cap D=(e_G, e_H),
$$
we have that $g_1=g_2$. Thus $P_H: D\lon H$ is injective.

Since the Lie algebra of the Lie subgroup $D$ is $Gr(B)\subseteq\g\ltimes_\phi\h$. Denote by ${P_H}_{*}: Gr(B)\lon\h$ the tangent map of $P_H$ at $(e_G, e_H)$, then we have
\begin{equation*}
{P_H}_{*}(B(u),u)=u,\quad \forall u\in\h,
\end{equation*}
which is an isomorphism from the vector space $Gr(B)$ to the vector
space $\h$. By the inverse function theorem, there exists an open set $V\subseteq D$ contains $(e_G, e_H)$ such that
$P_H: V\lon P_H(V)$ is a diffeomorphism. For any $(g, a)\in D$, we have that $$P_H((g, a)\cdot_\Phi V)=a\cdot_H\Phi(g)P_H(V)$$
is an open set. Denote by $W=a\cdot_H\Phi(g)P_H(V)$, we have
$$
P^{-1}_H|_{W}=L_{(g, a)}\circ P_H^{-1}|_{V}\circ\Phi(g^{-1})\circ l_{a^{-1}},
$$
where $l_{a^{-1}}(b)=a^{-1}\cdot_H b$ and $L_{(g, a)}(h, c)=(g, a)\cdot_{\Phi}(h, c)$.
Thus $P^{-1}_H$ is a smooth map in a neighborhood of $a\in H$. We have that $P^{-1}_H: H\lon D$ is a smooth map, it means that $P_H: D\lon H$ is a diffeomorphism. Thus there exists a smooth map $\huaB: H\lon G$ such that $D=Gr(\huaB)$, where
\begin{equation}\label{eqrbnew}
\huaB(h)=P_G(P_H^{-1}(h)), \quad \forall h\in H.
\end{equation}
For any $h_1, h_2\in H$, we have
$$
\Big(\huaB(h_1), h_1\Big)\cdot_{\Phi}\Big(\huaB(h_2), h_2\Big)=\Big(\huaB(h_1)\cdot_G\huaB(h_2), h_1\cdot_H\Phi(\huaB(h_1))h_2\Big)\in Gr(\huaB),
$$
which implies that
$$\huaB(h_1)\cdot_G\huaB(h_2)=\huaB(h_1\cdot_H\Phi(\huaB(h_1))h_2).$$
Thus $\huaB:H\lon G$ is a relative Rota-Baxter operator.
\end{proof}
\begin{rmk}
See \cite{HS} for more details of when a Lie subgroup is a closed Lie subgroup.
\end{rmk}

Denote by $(G\ltimes_\Phi H)/D=\{(g, h)D|\forall (g, h)\in G\ltimes_\Phi H\}$ the set of left cosets of $D$ in $G\ltimes_\Phi H$. Define a map $q_D: G\ltimes_\Phi H\lon (G\ltimes_\Phi H)/D$ by
$$
q_D(g, h)=(g, h)D, \quad \forall (g, h)\in G\ltimes_\Phi H.
$$
\begin{cor}
With the above notations, a relative Rota-Baxter Lie algebra $(\g, \h, \phi, B)$ can be integrated into a relative Rota-Baxter Lie group $(G, H, \Phi, \huaB)$ if and only if
\begin{itemize}
\item [\rm(i)]$D$ is a closed Lie subgroup;
\item [\rm(ii)]$q_D: G\times e_H\lon (G\ltimes_\Phi H)/D$ is bijective.
\end{itemize}
\end{cor}
\begin{proof}
For Lie groups $(G\ltimes_\Phi H, \cdot_\Phi)$ and $D$, $(G\ltimes_\Phi H, D, G\times e_H)$ is a double Lie group if and only if $D$ is a closed Lie subgroup and $q_D: G\times e_H\lon (G\ltimes_\Phi H)/D$ is bijective. Thus a relative Rota-Baxter Lie algebra $(\g, \h, \phi, B)$ can be integrated into a relative Rota-Baxter Lie group $(G, H, \Phi, \huaB)$ if and only if $D$ is a closed Lie subgroup and $q_D: G\times e_H\lon (G\ltimes_\Phi H)/D$ is bijective.
\end{proof}

\subsection{Integrability via  matched pairs of Lie groups}

In this subsection, we give another necessary and sufficient condition of a relative Rota-Baxter Lie algebra being  global integrable using matched pairs of Lie groups.

\begin{defi}\cite{Maj, Tak}
A matched pair of Lie algebras consists of a pair of Lie algebras $(\g,[\cdot,\cdot]_\g)$ and $(\h,[\cdot,\cdot]_\h)$, a representation $\rho: \g\to\gl(\h)$ of $\g$ on $\h$ and a representation $\mu: \h\to\gl(\g)$ of $\h$ on $\g$ such that
\begin{eqnarray*}
\label{eq:mp1}\rho(x) [u,v]_{\h}&=&[\rho(x)u,v]_{\h}+[u,\rho(x) v]_{\h}+\rho\big((\mu(v)x\big)u-\rho\big(\mu(u)x\big)v,\\
\label{eq:mp2}\mu(u) [x, y]_{\g}&=&[\mu(u)x,y]_{\g}+[x,\mu(u)y]_{\g}+\mu\big(\rho(y)u\big)x-\mu\big(\rho(x)u\big)y,
\end{eqnarray*}
 for all $x,y\in \g$ and $u,v\in \h$. Denote a matched pair of Lie algebras by $(\g,\h;\rho,\mu)$.
\end{defi}
The following result is well known.
\begin{pro}\label{mglie}
Let $(\g,\h;\rho,\mu)$ be a matched pair of Lie algebras. Then $(\g\oplus\h, [\cdot, \cdot]_{\bowtie})$ is a Lie algebra, where
\begin{equation*}
[(x, u), (y, v)]_{\bowtie}=([x, y]_\g+\mu(u)y-\mu(v)x, [u, v]_\h+\rho(x)v-\rho(y)u),
\end{equation*}
for all $x, y\in\g, u, v\in\h$.
\end{pro}

\begin{defi}\cite{Tak}
A matched pair of Lie groups consists of a pair of Lie groups $(G, \cdot_G), (H, \cdot_H)$, a left action of the Lie group $(G, \cdot_G)$ on the manifold $H$ and a right action of the Lie group $(H, \cdot_H)$ on the manifold $G$
\begin{equation*}
\Phi: G\times H\to H,\qquad (g,h)\mapsto  \Phi(g)h;\qquad \Theta:H\times G\to G\qquad (h,g)\mapsto \Theta(h)g,
\end{equation*}
such that
\begin{eqnarray*}
\label{mpg1} \Phi(g)(h_1\cdot_H h_2)&=&\Phi(g)h_1\cdot_H\Phi(\Theta(h_1)g)h_2;\\
\label{mpg2} \Theta(h)(g_1\cdot_G g_2)&=&\Theta(\Phi(g_2)h)g_1\cdot_G\Theta(h)g_2.
\end{eqnarray*}
Denote a matched pair of Lie groups by $(G,H;\Phi,\Theta)$.
\end{defi}
It is well known that there is a Lie group structure on $(G\times H, \cdot_{\bowtie})$ associated with $(G,H;\Phi,\Theta)$, such that $(G, \cdot_G)$ and $(H, \cdot_H)$ are Lie subgroups,
 where the group structure $\cdot_{\bowtie}$ is given by
\begin{equation*}
(g_1, h_1)\cdot_{\bowtie}(g_2, h_2)=(\Theta(h_2)g_1\cdot_G g_2, h_1\cdot_H\Phi(g_1)h_2),
\end{equation*}
for all $g_1, g_2\in G, h_1, h_2\in H$.

\begin{defi}
Let $\g$ and $\h$ be Lie algebras. Let $G$ and $H$ be connected and simply connected Lie groups integrating $\g$ and $\h$. A matched pair of Lie algebras $(\g,\h;\rho,\mu)$ is called integrable if there is a matched pair of Lie groups $(G, H; \Phi, \Theta)$ such that the Lie algebra of the Lie group $(G\times H, \cdot_{\bowtie})$ associated to $(G, H; \Phi, \Theta)$ is the Lie algebra $(\g\oplus\h, [\cdot, \cdot]_{\bowtie})$ associated to $(\g,\h;\rho,\mu)$.
\end{defi}

Any matched pair of Lie groups gives rise to a matched pair of Lie algebras. However,  matched pairs of Lie algebras are not always integrable, more details can be found in \cite{Maj}.

By Proposition \ref{graphro}, it shows that $(Gr(B), [\cdot, \cdot]_\phi)$ is a subalgebra of the Lie algebra $(\g\ltimes_\phi\h, [\cdot, \cdot]_{\phi})$. Then we have the following proposition.

\begin{pro}\label{matched}
Let $(\g, \h, \phi, B)$ be a relative Rota-Baxter Lie algebra. Then $(\g, Gr(B); \bar{\phi}, \bar{\theta})$ is a matched pair of Lie algebras $(\g, [\cdot,\cdot]_\g)$ and $(Gr(B), [\cdot, \cdot]_\phi)$, where
$$
\bar{\phi}:\g\lon\gl(Gr(B))\quad  \text{and} \quad \bar{\theta}:Gr(B)\lon\gl(\g)
$$
are given by
\begin{eqnarray*}
\bar{\phi}(x)(B(v), v)&=&(B(\phi(x)v), \phi(x)v), \\
\bar{\theta}(B(v), v)(x)&=&B(\phi(x)v)+[B(v), x]_\g,\quad \forall x\in\g, v\in\h.
\end{eqnarray*}
In other words, $(\g\oplus Gr(B), [\cdot,\cdot]_{\bowtie})$ is a Lie algebra, where $[\cdot,\cdot]_{\bowtie}$ is given by
\begin{eqnarray*}
&&[\Big(x, (B(u), u)\Big), \Big(y, (B(v), v)\Big)]_{\bowtie}\\
&=&\Big([x, y]_\g+\bar{\theta}(B(u), u)y-\bar{\theta}(B(v), v)x, \bar{\phi}(x)(B(v), v)-\bar{\phi}(y)(B(u), u)+[(B(u), u), (B(v), v)]_\phi\Big).
\end{eqnarray*}
\end{pro}
\begin{proof}
By \cite[Proposition 2.7, Proposition 2.8]{JSZ}, it follows that $\bar{\theta}:Gr(B)\lon\gl(\g)$ is a representation of $Gr(B)$ on $\g$ and
\begin{eqnarray*}
&&\bar{\phi}(x)[(B(u), u), (B(v), v)]_\phi\\
&=&[\bar{\phi}(x)(B(u), u), (B(v), v)]_\phi+[(B(u), u), \bar{\phi}(x)(B(v), v)]_\phi+\bar{\phi}(\bar{\theta}(B(v), v)x)(B(u), u)\\
&&-\bar{\phi}(\bar{\theta}(B(u), u)y)(B(v), v).
\end{eqnarray*}
Since
\begin{eqnarray*}
&&\bar{\phi}([x, y]_\g)(B(u), u)\\
&=&(B(\phi([x, y]_\g)u), \phi([x, y]_\g)u)\\
&=&(B(\phi(x)\phi(y)u), \phi(x)\phi(y)u)-(B(\phi(y)\phi(x)u), \phi(y)\phi(x)u)\\
&=&\bar{\phi}(x)\bar{\phi}(y)(B(u), u)-\bar{\phi}(y)\bar{\phi}(x)(B(u), u),
\end{eqnarray*}
it follows that $\bar{\phi}$ is a representation of $\g$ on $Gr(B)$.

Moreover, for any $x, y\in\g, u\in\h$, we have
\begin{eqnarray*}
&&[\bar{\theta}(B(u), u)x, y]_\g+[x, \bar{\theta}(B(u), u)y]_\g+\bar{\theta}\Big(\bar{\phi}(y)(B(u), u)\Big)x-\bar{\theta}\Big(\bar{\phi}(x)(B(u), u)\Big)y\\
&=&[B(\phi(x)u)+[B(u), x]_\g, y]_\g+[x, B(\phi(y)u)+[B(u), y]_\g]_\g+[B(\phi(y)u), x]_\g\\
&&+B(\phi(x)\phi(y)u)-[B(\phi(y)u), y]_\g-B(\phi(y)\phi(x)u)\\
&=&B(\phi([x, y]_\g)u)+[B(u), [x, y]_\g]_\g\\
&=&\bar{\theta}(B(u), u)[x, y]_\g,
\end{eqnarray*}
which implies that $(\g, Gr(B); \bar{\phi}, \bar{\theta})$ is a matched pair of Lie algebras.
\end{proof}

Define a linear map $\kappa: \g\oplus Gr(B)\lon \g\ltimes_\phi\h$ by
\begin{equation}\label{eqisomlm}
\kappa\Big(x, (B(u), u)\Big)=\Big(x+B(u), u\Big), \quad \forall x\in\g, u\in\h.
\end{equation}

\begin{pro}\label{iso}
With the above notations, $\kappa: \g\oplus Gr(B)\lon \g\ltimes_\phi\h$ is an isomorphism of Lie algebras.
\end{pro}
\begin{proof}
It is obvious that $\kappa:\g\oplus Gr(B)\lon \g\ltimes_\phi\h$ is bijective. For any $x,y\in\g, u, v\in\h$, we have
\begin{eqnarray*}
&&\kappa\Big([\Big(x, (B(u), u)\Big), \Big(y, (B(v), v)\Big)]_{\bowtie}\Big)\\
&=&\kappa\Big([x, y]+\bar{\theta}(B(u), u)y-\bar{\theta}(B(v), v)x, ~~\bar{\phi}(x)(B(v), v)-\bar{\phi}(y)(B(u), u)+[(B(u), u), (B(v), v)]_\phi\Big)\\
&=&\Big([x, y]_\g+[B(u), y]_\g+[x, B(v)]_\g+[B(u), B(v)]_\g, \phi(x)v-\phi(y)u+[u, v]+\phi(B(u))v-\phi(B(v))u\Big)\\
&=&[(x+B(u), u), (y+B(v), v)]_\phi\\
&=&[\kappa\Big(x, (B(u), u)\Big), \kappa\Big(y, (B(v), v)\Big)]_\phi.
\end{eqnarray*}
Thus $\kappa: \g\oplus Gr(B)\lon\g\ltimes_\phi\h$ is an isomorphism of Lie algebras.
\end{proof}
\emptycomment{
Let $(G, \cdot_G)$ and $(H, \cdot_H)$ be Lie groups, $\Phi:G\lon\Aut(H)$ be an action of $G$ on $H$. Consider the semidirect product Lie group $G\ltimes_{\Phi} H$, with multiplication $\cdot_{\Phi}$ given by
$$
(g_1, h_1)\cdot_\Phi(g_2, h_2)=(g_1\cdot_\Phi g_2, h_1\cdot_\Phi\Phi(g_1)h_2), \quad \forall g_i\in G, h_i\in H, i=1, 2.
$$
}
Let $(G, H, \Phi, \huaB)$ be a relative Rota-Baxter Lie group. Denote by $(\g, \h, \phi, B)$ the relative Rota-Baxter Lie algebra of $(G, H, \Phi, \huaB)$. By Proposition \ref{graphgroup},  $(Gr(\huaB), \cdot_\Phi)$ is a Lie group, where
$$
(\huaB(h_1), h_1)\cdot_\Phi(\huaB(h_2), h_2)=(\huaB(h_1\Phi(\huaB(h_1))h_2), h_1\cdot_H\Phi(\huaB(h_1))h_2), \quad \forall h_1, h_2\in H.
$$
\begin{pro}\label{matgroup}
With the above notations, $(G, Gr(\huaB); \bar{\Phi}, \bar{\Theta})$ is a matched pair of Lie groups, where
\begin{eqnarray*}
\bar{\Phi}(g)(\huaB(h), h)&=&(\huaB(\Phi(g)h), \Phi(g)h), \\
\bar{\Theta}(\huaB(h), h)g&=&(\huaB(\Phi(g)h))^{-1}\cdot_G g\cdot_G\huaB(h), \quad \forall g\in G, h\in H.
\end{eqnarray*}
Moreover, the matched pair of Lie algebras of $(G, Gr(\huaB); \bar{\Phi}, \bar{\Theta})$ is $(\g, Gr(B); \bar{\phi}, \bar{\theta})$, which is constructed in Proposition \ref{matched}.
\end{pro}
\begin{proof}
By \cite[Theorem 3.8]{JSZ}, it follows that $\bar{\Theta}$ is a right action of $Gr(\huaB)$ on the manifold $G$. It is straightforward to deduce that $\bar{\Phi}$ is a left action of $G$ on the manifold $Gr(\huaB)$. For any $g_1, g_2\in G$ and $h_1, h_2\in H$, we have
\begin{eqnarray*}
&&\bar{\Phi}(g_1)(\huaB(h_1), h_1)\cdot_{\Phi}\bar{\Phi}(\bar{\Theta}(\huaB(h_1), h_1)g_1)(\huaB(h_2), h_2)\\
&=&\Big(\huaB(\Phi(g_1)h_1), \Phi(g_1)h_1\Big)\cdot_\Phi\\
&&\Big(\huaB(\Phi(\huaB(\Phi(g_1)h_1)^{-1}\cdot_G g_1\cdot_G\huaB(h_1))h_2),\Phi(\huaB(\Phi(g_1)h_1)^{-1}\cdot_G g_1\cdot_G\huaB(h_1))h_2\Big)\\
&=&\Big(\huaB\big(\Phi(g_1)h_1\cdot_H\Phi(g_1\cdot_G\huaB(h_1))h_2\big), \Phi(g_1)h_1\cdot_H\Phi(g_1\cdot_G\huaB(h_1))h_2\Big)\\
&=&\Big(\huaB(\Phi(g_1)(h_1\cdot_H\Phi(\huaB(h_1))h_2)), \Phi(g_1)(h_1\cdot_H\Phi(\huaB(h_1))h_2)\Big)\\
&=&\bar{\Phi}(g_1)\Big((\huaB(h_1), h_1)\cdot_{\Phi}(\huaB(h_2), h_2)\Big),
\end{eqnarray*}
and
\begin{eqnarray*}
&&\bar{\Theta}(\bar{\Phi}((g_2)\huaB(h_1), h_1))g_1\cdot_G\bar{\Theta}(\huaB(h_1), h_1)g_2\\
&=&\huaB(\Phi(g_1\cdot_G g_2)h_1)^{-1}\cdot_G g_1\cdot_G\huaB(\Phi(g_2)h_1)\cdot_G(\huaB(\Phi(g_2)h_1))^{-1}\cdot_G g_2\cdot_G\huaB(h_1)\\
&=&\bar{\Theta}(\huaB(h_1), h_1)(g_1\cdot_G g_2).
\end{eqnarray*}
Thus, $(G, Gr(\huaB); \bar{\Phi}, \bar{\Theta})$ is a matched pair of Lie groups.

Moreover, $(\g, [\cdot, \cdot]_\g)$ and $(Gr(B), [\cdot, \cdot]_{\phi})$ are Lie algebras of $(G, \cdot_G)$ and $(Gr(\huaB), \cdot_\Phi)$ respectively. Denote by $\exp, \exp_G$ and $\exp_H$ the exponential map of $(G\times Gr(\huaB), \cdot_{\bowtie}), (G, \cdot_G)$ and $(H, \cdot_H)$ respectively. For any $x\in\g, (B(u), u)\in Gr(B)$, we have
\begin{eqnarray*}
&&\frac{d}{ds}\frac{d}{dt}|_{t=0, s=0}\bar{\Phi}(\exp(sx))\exp(tB(u), tu)\\
&=&\frac{d}{ds}\frac{d}{dt}|_{t=0, s=0}\Big(\huaB(\Phi(\exp_G(sx))\exp_H(tu)), \Phi(\exp_G(sx))\exp_H(tu)\Big)\\
&=&(B(\phi(x)u), \phi(x)u)\\
&=&\bar{\phi}(x)(B(u), u),
\end{eqnarray*}
and
\begin{eqnarray*}
&&\frac{d}{ds}\frac{d}{dt}|_{t=0, s=0}\bar{\Theta}(\exp(sB(-u), -su))\exp(tx)\\
&=&\frac{d}{ds}\frac{d}{dt}|_{t=0, s=0}\huaB(\Phi(\exp_G(tx))\exp_H(-su))^{-1}\cdot_G\exp_G(tx)\cdot_G\huaB(\exp_H(-su))\\
&=&\frac{d}{ds}\frac{d}{dt}|_{t=0, s=0}\huaB(\exp_H(-su))^{-1}\cdot_G\exp_G(tx)\cdot_G\huaB(\exp_H(-su))\\
&&+\frac{d}{ds}\frac{d}{dt}|_{t=0, s=0}\huaB(\Phi(\exp_G(tx))\exp_H(-su))^{-1}\cdot_G\huaB(\exp_H(-su))\\
&=&[B(u), x]_\g+\frac{d}{dt}\frac{d}{ds}|_{t=0, s=0}\huaB(\Phi(\exp_G(tx))\exp_H(-su))^{-1}\cdot_G\huaB(\exp_H(-su))\\
&=&[B(u), x]_\g+B(\phi(x)u)\\
&=&\bar{\theta}(B(u),u)x.
\end{eqnarray*}
Thus, the matched pair of Lie algebras of $(G, Gr(\huaB); \bar{\Phi}, \bar{\Theta})$ is $(\g, Gr(B); \bar{\phi}, \bar{\theta})$.
\end{proof}

Let $(\g, \h, \phi, B)$ be a relative Rota-Baxter Lie algebra. Let $G$ and $H$ be the connected and simply connected Lie
groups integrating $\g$ and $\h$ respectively, and $\Phi: G\to \Aut(H)$ be the
integrated action. By Proposition \ref{graphro},   $(Gr(B), [\cdot, \cdot]_\phi)$ is a Lie subalgebra of $\g\ltimes_\phi\h$. Denote by $E$ the connected and simply connected Lie group corresponding to $Gr(B)$. By Proposition \ref{matched},   $(\g, Gr(B); \bar{\phi}, \bar{\theta})$ is a matched pair of Lie algebras $(\g, [\cdot,\cdot]_\g)$ and $(Gr(B), [\cdot, \cdot]_\phi)$.

\begin{thm}\label{intthm}
With the above notations, a relative Rota-Baxter Lie algebra $(\g, \h, \phi, B)$ can be integrated into a relative Rota-Baxter Lie group $(G, H, \Phi, \huaB)$ if and only if the matched pair $(\g, Gr(B); \bar{\phi}, \bar{\theta})$ can be integrated into a matched pair of Lie groups $(G, E; \hat{\Phi}, \hat{\Theta})$.
\end{thm}
\begin{proof}
If the matched pair of Lie algebras $(\g, Gr(B); \bar{\phi}, \bar{\theta})$ can be integrated into a matched pair of Lie groups $(G, E; \hat{\Phi}, \hat{\Theta})$, we denote by $\cdot_{\bowtie}$ the group product on $G\times E$ associated with $(G, E; \hat{\Phi}, \hat{\Theta})$. By Proposition \ref{iso}, the map $\kappa: \g\oplus Gr(B)\lon\g\ltimes_\phi\h$ is a Lie algebra isomorphism and
$$
\kappa(x, 0)=(x, 0),\quad \kappa\Big(0, (B(u), u)\Big)=(B(u), u), \quad \forall x\in\g, u\in\h.
$$
Thus there is a unique Lie group isomorphism $$K: G\times E\lon G\ltimes_\Phi H,$$ such that $K_{*(e_{G}, e_{E})}=\kappa$. Since $G\times e_E$ is a connected closed Lie subgroup and $\kappa|_{\g\times 0}=\Id$, we have $K|_{G\times e_E}=\Id$. Moreover, since $E$ is the integration of $Gr(B)$, and by the fact that $\kappa\Big(0, (B(u), u)\Big)=(B(u), u)$, it follows that the Lie algebra of the Lie subgroup $K(e_G\times E)$ is
$$
\kappa(0\times Gr(B))=Gr(B)\subseteq\g\ltimes_\phi\h.
$$
By the fact that $e_G\times E$ is a closed set of $G\times E$, $$(G\times e_E)\cap (e_G\times E)=(e_G, e_E),\quad (e_G\times E)\cdot (G\times e_E)=G\times E,$$ and $$K\Big((e_G\times E)\cdot( G\times e_E)\Big)=K(e_G\times E)\cdot_{\ltimes} (G\times e_E),$$ it follows that
$$
(G\times e_E)\cap K(e_G\times E)=(e_G, e_H)\quad \text{and} \quad K(e_G\times E)\cdot_{\ltimes}(G\times e_E)=G\ltimes_{\Phi} H.
$$
By Theorem \ref{proint2}, it follows that the relative Rota-Baxter operator $B:\h\lon\g$ can be integrated into a relative Rota-Baxter operator $\huaB:H\lon G$.

Conversely, if a relative Rota-Baxter Lie algebra $(\g, \h, \phi, B)$ can be integrated into a relative Rota-Baxter Lie group $(G, H, \Phi, \huaB)$, by Proposition \ref{graphgroup} and $Gr(\huaB)$ is diffeomorphic with $H$, we have that $Gr(\huaB)$ is a connected and simply connected Lie group integrating $Gr(B)$. By Proposition \ref{matgroup}, it follows that $(G, Gr(\huaB); \hat{\Phi}, \hat{\Theta})$ is a matched pair of Lie groups integrating $(\g, Gr(B); \bar{\phi}, \bar{\theta})$, where
\begin{eqnarray*}
\bar{\Phi}(g)(\huaB(h), h)&=&(\huaB(\Phi(g)h), \Phi(g)h), \\
\bar{\Theta}(\huaB(h), h)g&=&(\huaB(\Phi(g)h))^{-1}\cdot_G\g\cdot_G\huaB(h),
\end{eqnarray*}
for all $\g\in G, h\in H$.
\end{proof}

\begin{cor}\label{Cor1}
With above notations, if $G$ or $E$ is compact, a relative Rota-Baxter Lie algebra $(\g, \h, \phi, B)$ can be integrated to a relative Rota-Baxter Lie group $(G, H, \Phi, \huaB)$.
\end{cor}
\begin{proof}
Since $G$ or $E$ is connected and simply connected Lie groups and $G$ or $E$ is compact, by \cite[Theorem 4.2]{Maj}, we have a matched pair of Lie group $(G, E; \hat{\Phi}, \hat{\Theta})$, such that the matched pair of Lie algebras of $(G, E; \hat{\Phi}, \hat{\Theta})$ is $(\g, Gr(B); \bar{\phi}, \bar{\theta})$. By Theorem \ref{intthm}, $B:\h\lon\g$ can be integrated to a relative Rota-Baxter operator on Lie groups.
\end{proof}

Let $(\g, \h, \phi, B)$ be a relative Rota-Baxter Lie algebra. By Proposition \ref{matched} and \ref{mglie}, it means that $(\g, Gr(B); \bar{\phi}, \bar{\theta})$ is a matched pair, and $(\g\oplus Gr(B), [\cdot, \cdot]_{\bowtie})$ is a Lie algebra. Denote by $T, G$ and $E$ the connected and simply connected Lie groups corresponding to $\g\oplus Gr(B), \g$ and $Gr(B)$ respectively. Define linear maps $\zeta_1: \g\lon\g\oplus Gr(B)$ and $\zeta_2: Gr(B)\lon\g\oplus Gr(B)$ by
$$
\zeta_1(x)=(x, 0), \quad\forall x\in\g, \quad\quad \zeta_2(B(u), u)=(0, (B(u), u)), \quad \forall (B(u), u)\in Gr(B),
$$
which are  Lie algebra homomorphisms. Denote by $\Xi_1: G\lon T$ and $\Xi_2: E\lon T$ the Lie group homomorphisms corresponding to $\zeta_1$ and $\zeta_2$. Define a smooth map $\Upsilon:G\times E\lon T$ by
$$
\Upsilon(g, f)=\Xi_1(g)\cdot\Xi_2(f), \quad \forall g\in G, f\in E.
$$
\begin{pro}\label{diffmap-1}
With the above notations, the relative Rota-Baxter Lie algebra $(\g, \h, \phi, B)$ is integrable if and only if $\Upsilon:G\times E\lon T$ is a diffeomorphism.
\end{pro}
\begin{proof}
If $(\g, \h, \phi, B)$ is integrated to $(G, H, \Phi, \huaB)$, then $E=Gr(\huaB)$. By Proposition \ref{matgroup}, we obtain that $(G, Gr(\huaB); \bar{\Phi}, \bar{\Theta})$ is a matched pair, and $(G\times E, \cdot_{\bowtie})$ is a Lie group. Moreover,
$$
\Xi_1(g)=(g, e_E), \quad \forall g\in G, \quad \Xi_2(\huaB(h), h)=(e_G, (\huaB(h), h)), \quad \forall (\huaB(h), h)\in E.
$$
Thus $\Upsilon(g, (\huaB(h), h))=(\huaB(\Phi(g)h)^{-1}\cdot_G g\cdot_G\huaB(h), (\huaB(\Phi(g)h), \Phi(g)h))$ is a diffeomorphism.

Conversely, if $\Upsilon:G\times E\lon T$ is a diffeomorphism, we can define a Lie group structure on $G\times E$ by
$$
(g_1, f_1)\cdot (g_2, f_2)=\Upsilon^{-1}(\Upsilon(g_1, f_1)\cdot\Upsilon(g_2, f_2)), \quad (g_i, f_i)\in G\times E.
$$
Moreover $G$ and $E$ are Lie subgroups of $G\times E$, which means that $G\times E$ is a matched pair of $G$ and $E$.
Since $\Upsilon_{*(e_G, e_E)}=\Id$ and $\Upsilon:G\times E\lon T$ is a Lie group homomorphism, then the Lie algebra structure $[\cdot, \cdot]$ on $\g\oplus Gr(B)$ is given by
$$
[(x, (B(u), u)), (y, (B(v), v))]=[(x, (B(u), u)), (y, (B(v), v))]_{\bowtie}.
$$
Thus it follows that $(\g, \h, \phi, B)$ is integrable from Theorem \ref{intthm}.
\end{proof}

Given a relative Rota-Baxter Lie algebra $(\g, \h, \phi, B)$, there is a double Lie groupoid (see for example \cite[Definition 3.1]{TangM}) given by the square
\begin{equation} \label{eq:double}
     \xymatrix{\Gamma\ar@<-0.5 ex>[r]_{s^h}\ar@<0.5 ex>[r]^{t^h}\ar@<-0.5 ex>[d]_{t^v}\ar@<0.5 ex>[d]^{s^v}& E\ar@<-0.5 ex>[d]\ar@<0.5 ex>[d]\\
     G\ar@<-0.5 ex>[r]\ar@<0.5 ex>[r]&pt}
\end{equation}
where $$\Gamma=\{(h_2,a_2,a_1,h_1)\in E\times G\times G\times E | \Xi_2(h_2)\Xi_1(a_1)=\Xi_1(a_2)\Xi_2(h_1)\},$$ and the  horizontal source, target, and multiplication are given by
\begin{equation*}
     s^h(\xi)=h_1, \ t^h(\xi)=h_2, \ m^h(\xi, \xi')=(h_2, a_2a_2', a_1a_1', h_1'),
\end{equation*}
 and the vertical source, target, and multiplication are given by
  \begin{equation*}
     s^v(\xi)=a_1, \ t^v(\xi)=a_2, \ m^v(\xi, \xi')=(h_2h_2', a_2, a_1', h_1h_1'),
 \end{equation*}
 for $\xi=(h_2,a_2,a_1,h_1)$ and  $\xi'=(h_2',a_2',a_1',h_1')\in\Gamma$. The Artin-Mazur codiagonal $\bar{\Gamma}_\bullet$ of the double groupoid \eqref{eq:double} is given  by
  $$
\bar{\Gamma}_\bullet=\cdots G\times\Gamma\times E\aaar G\times E\rightrightarrows pt.
 $$
As proven in \cite{TangM}, the Artin-Mazur codiagonal of a double Lie groupoid is always a local Lie 2-groupoid. Thus $(\bar{\Gamma}_\bullet, d_\bullet, s_\bullet)$ is a local Lie $2$-group. Here for the concept of local Lie $n$-groupoids and local Lie $n$-groups\footnote{A local Lie $n$-group is a pointed local Lie groupoid.}, we refer to \cite[Definition 1.1]{Zhu11}.

At the end of this subsection, we will show that if $(\g, \h, \phi, B)$ is integrable, then $\bar{\Gamma}_\bullet$ is a Lie $1$-group, that is the nerve of a Lie group. Moreover, even if $(\g, \h, \phi, B)$ is not integrable, $\bar{\Gamma}_\bullet$ is still Morita equivalent to a Lie 1-group  $NT_\bullet$ as local Lie 2-groups.

We first recall the definition of Morita equivalences of local Lie $n$-groupoids from \cite{CZhu}. Let $(X_\bullet, d_\bullet, s_\bullet)$ be a simplicial manifold. The boundary space or $i$-matching space of $X_\bullet$ is
$$
\partial_i(X_\bullet)=\{(x_0, \cdots, x_i)\in X_{i-1}^{i+1}|d_p(x_q)=d_q(x_p), \forall p<q\}.
$$
Let $X_\bullet$ and $Y_\bullet$ be  local Lie $n$-groupoids and $f_\bullet: X_\bullet\lon Y_\bullet$ be a simplicial morphism. If the following map
$$
q_i=((d_0, \cdots, d_i), f_i):~~ X_i\lon\partial_i(X_\bullet)\times_{\partial_i(Y_\bullet)}Y_i
$$
is submersion for $0\leq i\leq n-1$ and injective etale when $i=n$, then  $f_\bullet: X_\bullet\lon Y_\bullet$ is called a hypercover.
\begin{defi}\cite{CZhu}
Two local Lie $n$-groupoids $X_\bullet$ and $Y_\bullet$ are Morita equivalent if there is another local Lie $n$-groupoid $Z_\bullet$ and hypercovers $f_\bullet: Z_\bullet\lon X_\bullet, g_\bullet: Z_\bullet\lon Y_\bullet$.
\end{defi}
Denote by $NT_\bullet$ the never of $T$. Define smooth maps $\Upsilon_\bullet:\bar{\Gamma}_\bullet\lon NT_\bullet$ by
\begin{eqnarray*}
\Upsilon_1(a, h)&=&\Upsilon(a, h), \quad \forall (a, h)\in G\times E\\
\Upsilon_2(a_3, (h_3,a_2,a_1,h_2), h_1)&=&(\Upsilon(a_3, h_3),\Upsilon(a_1, h_1)), \quad \forall a_3\in G, h_1\in E, (h_3,a_2,a_1,h_2)\in\Gamma.
\end{eqnarray*}
\begin{pro}\label{hp1}
With the above notations, $\Upsilon_\bullet:\bar{\Gamma}_\bullet\lon NT_\bullet$ is a hypercover of local Lie $2$-groups, $\bar{\Gamma}_\bullet$ and $NT_\bullet$ are Morita equivalent. In particular, if $(\g, \h, \phi, B)$ is integrable, then there are diffeomorphisms $\Upsilon_\bullet:\bar{\Gamma}_\bullet\lon NT_\bullet$, thus $\bar{\Gamma}_\bullet$ is a Lie $1$-group.
\end{pro}
\begin{proof}
The boundary spaces are $$\partial_1(\bar{\Gamma}_\bullet)=pt\times pt, \quad
\partial_2(\bar{\Gamma}_\bullet)=(G\times E)\times(G\times E)\times(G\times E).
$$
In our case, the map $q_1: G\times E\lon NT_1$ is simply $\Upsilon_1$ and $q_2: \bar{\Gamma}_2\lon(G\times E)\times(G\times E)\times(G\times E)\times_{T\times T\times T}NT_2$ is given by
$$
q_2(a_3, (h_3,a_2,a_1,h_2), h_1)=\Big((a_1, h_1), (a_3a_2, h_2h_1), (a_3, h_3), \Upsilon_1(a_3, h_3), \Upsilon_1(a_1, h_1)\Big).
$$
Since $(\Upsilon_{1})_{*(a, h)}=-L_{\Xi_1(a)*}R_{\Xi_2(h)*}$ is invertible,  $\Upsilon_1$ is a local diffeomorphism thus in particular a submersion.
By a  direct calculation, $q_2$ is a diffeomorphism with inverse map $q^{-1}_2:(G\times E)\times(G\times E)\times(G\times E)\times_{T\times T\times T}NT_2\lon \bar{\Gamma}_2$ given by
$$
q_2^{-1}\Big((a_1, h_1), (a_2, h_2), (a_3, h_3), t_1, t_2\Big)=(a_3, (h_3, a_3^{-1}a_2, a_1, h_2h_1^{-1}), h_1).
$$
Therefore $q_2$ is in particular injective etale. Thus $\Upsilon_\bullet:\bar{\Gamma}_\bullet\lon NT_\bullet$ is a hypercover of local Lie $2$-groups.
\emptycomment{
Thus we have the following picture
\begin{equation*}
    \xymatrix{&\bar{\Gamma}_\bullet\ar[dr]^{\Upsilon_\bullet}\ar[dl]_{id_\bullet}&&NT_\bullet\ar[dr]^{id_\bullet}\ar[dl]_{\iota_\bullet}\\
    \bar{\Gamma}_\bullet&&NT_\bullet&&NT_\bullet}
\end{equation*}
By \cite[Lemma 2.7]{Zhu11}, we have
\begin{equation*}
    \xymatrix{&&\bar{\Gamma}_\bullet\times_{NT_\bullet}NT_\bullet \ar[dl]_{p^1_\bullet}\ar[dr]^{p^2_\bullet}\\
    &\bar{\Gamma}_\bullet\ar[dr]^{\Upsilon_\bullet}\ar[dl]_{id_\bullet}&&NT_\bullet\ar[dr]^{id_\bullet}\ar[dl]_{\iota_\bullet}\\
    \bar{\Gamma}_\bullet&&NT_\bullet&&NT_\bullet}
\end{equation*}
By \cite[Lemma 2.6]{Zhu11},   $p^1_\bullet$ and $p^2_\bullet$ are hypercovers.}
Therefore,  $\bar{\Gamma}_\bullet$ and $NT_\bullet$ are Morita equivalent.

If $(\g, \h, \phi, B)$ is integrable, by Proposition \ref{diffmap-1}, $\Upsilon_1$ is a diffeomorphism. By \cite[Lemma 2.5]{Zhu11}, $\bar{\Gamma}_\bullet$ is isomorphic to $NT_\bullet$, which implies that $\bar{\Gamma}_\bullet$ is a Lie $1$-group.
\end{proof}

\begin{rmk}
As we have seen, when $(\g, \h, \phi, B)$ is integrable, the local Lie $2$-group $\bar{\Gamma}_\bullet$ is diffeomorphic to $NT_\bullet$. Thus when  $(\g, \h, \phi, B)$ is not integrable, the double groupoid \eqref{eq:double} or its Artin-Mazur codiagonal $\bar{\Gamma}_\bullet$ can be viewed as a replacement of the integration.

  In \cite{JSZ}, it was proved that every Rota-Baxter operator on Lie algebras $B:\h\lon\g$ can be integrated into a local relative Rota-Baxter operator on Lie groups $\huaB: U(\subseteq H)\lon G$. Thus when $(\g, \h, \phi, B)$ is not integrable, the local relative Rota-Baxter operator $\huaB$ may also be viewed as a replacement of the integration.

  However we need to notice that the local Lie group $U$ coming from $\huaB$ is not $\bar{\Gamma}_\bullet$. In fact $\dim \bar{\Gamma}_1 = 2 \dim U$. These two sorts of replacements are different but related.
\end{rmk}

\subsection{Integrability via factorization of a diffeomorphism}

In this subsection, we give another necessary and sufficient condition of  Rota-Baxter Lie algebras being global integrable.

Let $(\g, B)$ be a Rota-Baxter Lie algebra. By Proposition \ref{graphro}, it shows that $(Gr(B), [\cdot,\cdot]_{\phi})$ is a Lie algebra. Define linear maps $f_{-}:Gr(B)\lon\g$ and $f_{+}: Gr(B)\lon\g$ by
$$
f_{-}(B(x), x)=B(x), \quad f_{+}(B(x), x)=B(x)+x, \quad \forall (B(x), x)\in Gr(B).
$$
Since
\begin{eqnarray*}
f_{-}\Big([(B(x), x), (B(y), y)]_{\phi}\Big)&=&f_{-}([B(x), B(y)]_\g, [x, B(y)]_\g+[B(x), y]_\g+[x, y]_\g)\\
&=&[B(x), B(y)]_\g=[f_{-}(B(x), x), f_{-}(B(y), y)]_\g,
\end{eqnarray*}
and
\begin{eqnarray*}
f_{+}\Big([(B(x), x), (B(y), y)]_{\phi}\Big)&=&B(x), B(y)]_\g+[x, B(y)]_\g+[B(x), y]_\g+[x, y]_\g\\
&=&[f_{+}(B(x), x), f_{+}(B(y), y)]_\g,
\end{eqnarray*}
it follows that $f_{-}$ and $f_{+}$ are Lie algebra homomorphisms.

Denote by $D$ and $G$ connected and simply connected Lie groups integrating $Gr(B)$ and $\g$ respectively. Let $F_{-}: D\lon G$ and $F_{+}: D\lon G$ be Lie groups homomorphisms integrating $f_{-}$ and $f_{+}$ respectively. Define a smooth map $J: D\lon G$ by
$$
J(a)=F_{+}(a)\cdot_G F_{-}(a)^{-1}, \quad \forall a\in D.
$$
\begin{thm}\label{thmrbint}
With the above notations, the Rota-Baxter operator $B:\g\lon\g$ is integrable if and only if the map $J: D\lon G$ is invertible.
\end{thm}
\begin{proof}
If $J: D\lon G$ is invertible, define a smooth map $\huaB: G\lon G$ by
$$
\huaB(g)=F_{-}(J^{-1}(g)), \quad \forall g\in G.
$$
Since
\begin{eqnarray*}
\huaB(g)\cdot_G\huaB(h)&=&F_{-}(J^{-1}(g))\cdot_G F_{-}(J^{-1}(h))\\
&=&F_{-}J^{-1}\Big(J((J^{-1}(g)\cdot_G J^{-1}(h)))\Big)\\
&=&\huaB(F_{+}(J^{-1}(g))\cdot_G F_{+}(J^{-1}(h))\cdot_G F_{-}(J^{-1}(h))^{-1}\cdot_G F_{-}(J^{-1}(g))^{-1})\\
&=&\huaB(g\cdot_G F_{-}(J^{-1}(g))\cdot_G h\cdot_G F_{-}(J^{-1}(g))^{-1})\\
&=&\huaB(g\cdot_G\huaB(g)\cdot_G h\cdot_G\huaB(g)^{-1}),
\end{eqnarray*}
it follows that $\huaB: G\lon G$ is a Rota-Baxter operator. Moreover, the tangent map of $J$ at $e_D$ is
$$
J_{*}(B(x), x)=f_{+}(B(x), x)-f_{-}(B(x), x)=x, \quad \forall (B(x), x)\in Gr(B),
$$
we have that
$$
\huaB_{*}(x)=f_{-}(J^{-1}_{*}(x))=f_{-}(B(x), x)=B(x), \quad \forall x\in\g.
$$
Thus $B:\g\lon\g$ is integrable.

Conversely, if $B:\g\lon\g$ is integrable, denote by $\huaB: G\lon G$ the integration of $B$. By Proposition \ref{graphgroup}, $Gr(\huaB)$ is a connected and simply connected Lie group integrating $Gr(B)$, i.e. $D=Gr(\huaB)$. Define $F_{\pm}: D\lon G$ by
\begin{eqnarray*}
F_{-}(\huaB(g), g)&=&\huaB(g), \\
F_{+}(\huaB(g), g)&=&g\cdot_G\huaB(g), \quad \forall (\huaB(g), g)\in Gr(\huaB).
\end{eqnarray*}
We have that $F_{\pm}$ are Lie group homomorphisms and $(F_{\pm})_{*e_D}=f_{\pm}$.
Moreover, the smooth map $J:Gr(\huaB)\lon G$ is given by
$$
J(\huaB(g), g)=F_{+}(\huaB(g), g)\cdot_G F_{-}(\huaB(g), g)^{-1}=g,
$$
thus the smooth map $J: Gr(\huaB)\lon G$ is invertible.
\end{proof}

\begin{rmk}
If $B:\h\lon\g$ is an invertible relative Rota-Baxter operator on Lie algebras, then the inverse of $B$ is a crossed homomorphism on Lie algebras. In \cite{JS, JLS}, it was shown that every crossed homomorphism on Lie algebras is integrable. But it is not clear that whether every invertible relative Rota-Baxter operator on Lie algebras is integrable.
\end{rmk}

\section{Integration and factorizable Poisson-Lie groups}
In this section, we apply the integration of Rota-Baxter operators to characterize when Poisson-Lie groups integrating factorizable Lie bialgebras are factorizable Poisson-Lie groups.

We first recall  factorizable Lie bialgebras and factorizable Poisson-Lie groups. See \cite{Dri, KOS, LuW, RS} for more details.
\begin{defi}
Let $(\g, [\cdot, \cdot]_\g)$ be a Lie algebra and $\gamma:\g\lon\g\otimes\g$ be a linear map. The pair $(\g, \gamma)$ is called a Lie bialgebra if
\begin{itemize}
\item [\rm(i)] $\gamma$ is a $1$-cocycle with respect to the representation $(\g\otimes\g, \ad\otimes\Id+\Id\otimes\ad)$, i.e.
\begin{equation*}
\gamma([x, y]_\g)=[\gamma(x), y]_\g+[x, \gamma(y)]_\g, \quad \forall x, y\in\g.
\end{equation*}
\item [\rm(ii)] $(\g^*, [\cdot, \cdot]_{\g^*})$ is a Lie algebra, where
$$
\langle[\xi, \eta]_{\g^*}, x\rangle=\gamma(x)(\xi, \eta), \quad \forall \xi, \eta\in\g^*, ~~x\in\g.
$$
\end{itemize}
\end{defi}
Let $(\g, [\cdot, \cdot]_\g)$ be a Lie algebra. Denote by $$r=a+s=\sum_i x_i\otimes y_i\in\g\otimes\g, \quad \text{where}~~~~ a\in\wedge^2\g,~~ s\in S^2\g.$$ Define the linear map $\gamma:\g\lon\g\otimes\g$ by $\gamma=\dM r$, where $\dM$ is the Chevalley-Eilenberg coboundary operator of the Lie algebra $\g$ with coefficients in the representation $(\g\otimes\g, \ad\otimes\Id+\Id\otimes\ad)$. 
Then $(\g, \dM r)$ is a Lie bialgebra if and only if $\ad_{x}s=0$ and $\ad_{x}\Courant{r, r}=0$ for all $x\in\g$, where
\begin{eqnarray*}
\Courant{r, r}&=&[r_{12}, r_{13}]+[r_{12}, r_{23}]+[r_{13}, r_{23}]\\
&=&\sum_{i,j}[x_i,x_j]_\g\otimes y_i\otimes y_j+\sum_{i,j}x_i\otimes [y_i, x_j]_\g\otimes y_j+\sum_{i,j}x_i\otimes x_j\otimes [y_i, y_j]_\g.
\end{eqnarray*}
In particular, if $\Courant{r, r}=0$, the pair $(\g, \dM r)$ is called a quasi-triangular Lie bialgebra.

Let $(\g, \dM r)$ be a quasi-triangular Lie bialgebra. Define a linear map $I:\g^*\lon\g$ by $I=r_{+}-r_{-}$, where
$$
\langle r_{+}(\xi), \eta\rangle=r(\xi, \eta), \quad \langle r_{-}(\xi), \eta\rangle=-r(\eta, \xi), \quad \forall \xi, \eta\in\g^*.
$$
Then in this setting, the Lie bracket on $\g^*$ is given by
$$
[\xi, \eta]_{\g^*}=\ad^{*}_{r_{+}(\xi)}\eta-\ad^{*}_{r_{-}(\eta)}\xi, \quad\forall \xi, \eta\in\g^*.
$$
Moreover, $r_{+}:\g^*\lon\g$ and $r_{-}:\g^*\lon\g$ are Lie algebra homomorphisms.

\begin{defi}
The quasi-triangular Lie bialgebra $(\g, \dM r)$ is called a factorizable Lie bialgebra if the linear map $I:\g^*\lon\g$ is an isomorphism.
\end{defi}
A factorizable Lie bialgebra can give rise a Rota-Baxter Lie algebra.
\begin{thm}\cite{LS}
Let $(\g, \dM r)$ be a factorizable Lie bialgebra with $I=r_+-r_{-}$. Then $(\g, r_{-}\circ I^{-1})$ is a Rota-Baxter Lie algebra.
\end{thm}
\begin{defi}
Let $(G, \cdot_G)$ be a Lie group and $(G, \pi)$ is a Poisson manifold. The triple $(G, \cdot_G, \pi)$ is called a Poisson-Lie group if $\pi$ satisfies
$$
\pi_{g\cdot_G h}=L_g\pi_h+R_h\pi_g, \quad\forall g, h\in G.
$$
\end{defi}
It is well known that Lie bialgebras are one-to-one correspondence with connected and simply connected Poisson-Lie groups.

Let $(\g, \dM r)$ be a factorizable Lie algebra. Denote by $(G, \cdot_G, \pi_r)$ the connected and simply connected Poisson-Lie group integrating $(\g, [\cdot, \cdot]_\g)$, where the Poisson structure $\pi_r$ is given by
$$
\pi_r=r^{L}-r^{R}=a^{L}-a^{R},
$$
where $r^{L}$ and $r^{R}$ are vector fields obtained by left translation   and right translation of $r$ respectively.

Denote by $(G^*, \cdot_{G^*})$ the connected and simply connected Lie group integrating $(\g^*, [\cdot, \cdot]_{\g^*})$. Let $R_{+}: G^*\lon G$ and $R_{-}: G^*\lon G$ be the lifted Lie group homomorphisms of the Lie algebra homomorphisms $r_+:\g^*\lon\g$ and $r_-:\g^*\lon\g$.

Define a smooth map $J: G^*\lon G$ by
$$
J(q)=R_+(q)\cdot_G R_{-}(q)^{-1}, \quad \forall q\in G^*.
$$
The Poisson-Lie group $(G, \cdot_G, \pi_r)$ is called a {\bf factorizable Poisson-Lie group} if $J$ is a diffeomorphism. It is known that  not every factorizable Lie bialgebra can be integrated into a factorizable Poisson-Lie group. In the sequel,  we give a necessary and sufficient condition on  when a factorizable Lie bialgebra can be integrated into a factorizable Poisson-Lie group using the integrability of Rota-Baxter Lie algebras.

\begin{thm}
With the above notations, the Poisson-Lie group $(G, \cdot_G, \pi_r)$ integrating $(\g, \dM r)$ is a factorizable Poisson-Lie group if and only if the Rota-Baxter operator $r_{-}\circ I^{-1}:\g\lon\g$ is integrable.
\end{thm}
\begin{proof}
Since $r_{-}\circ I^{-1}:\g\lon\g$ is a Rota-Baxter operator, thus $(Gr(r_{-}\circ I^{-1}), [\cdot, \cdot]_{\ad})$ is a Lie algebra. Define a linear map $\hat{I}:\g^{*}\lon Gr(r_{-}\circ I^{-1})$ by
$$
\hat{I}(\xi)=(r_{-}(\xi), I(\xi)), \quad \forall \xi\in\g^*.
$$
We have
\begin{eqnarray*}
\hat{I}([\xi, \eta]_{\g^*})&=&([r_{-}(\xi), r_{-}(\eta)]_\g, [r_{+}(\xi), r_{+}(\eta)]_\g-[r_{-}(\xi), r_{-}(\eta)]_\g)\\
&=&([r_{-}(\xi), r_{-}(\eta)]_\g, [I(\xi), r_{-}(\eta)]_\g+[r_{-}(\xi), I(\eta)]_\g+[I(\xi), I(\eta)]_\g)\\
&=&[\hat{I}(\xi), \hat{I}(\eta)]_\ad.
\end{eqnarray*}
Thus $\hat{I}:\g^{*}\lon Gr(r_{-}\circ I^{-1})$ is a Lie algebra isomorphism. Moreover, we have two Lie algebra homomorphisms as following,
$$
r_{+}\circ\hat{I}^{-1}: Gr(r_{-}\circ I^{-1})\lon\g, \quad r_{+}\circ\hat{I}^{-1}(r_{-}(I^{-1}(x)), x)=r_{-1}(I^{-1}(x))+x,
$$
and
$$
r_{-}\circ\hat{I}^{-1}: Gr(r_{-}\circ I^{-1})\lon\g, \quad r_{-}\circ\hat{I}^{-1}(r_{-}(I^{-1}(x)), x)=r_{-}(I^{-1}(x)).
$$
Denote by $E, G$ and $G^*$ connected and simply connected Lie groups integrating $Gr(r_{-}\circ I^{-1}), \g$ and $\g^*$ respectively. Denote by $R_{-}, R_{+}$ and $\hat{\huaI}$ Lie group homomorphisms integrating $r_{-}, r_{+}$ and $\hat{I}$ respectively. We have that $G^*=\hat{\huaI}^{-1}(E)$. Define a smooth map $J: E\lon G$ by
$$
J(a)=R_{+}(\hat{\huaI}^{-1}(a))\cdot_G R_{-}(\hat{\huaI}^{-1}(a))^{-1}, \quad \forall a\in E.
$$
By Theorem \ref{thmrbint}, it follows that $r_{-}\circ I^{-1}$ is integrable if and only if $J$ is a diffeomorphism. Thus the Poisson-Lie group $(G, \cdot_G, \pi_r)$ integrating $(\g, \dM r)$ is a factorizable Poisson-Lie group if and only if the Rota-Baxter operator $r_{-}\circ I^{-1}:\g\lon\g$ is integrable.
\end{proof}

\section{Examples}
In this section, we give some examples. In particular, we study the integrability of Rota-Baxter operators on the nontrivial $2$-dimensional Lie algebra. It turns out that not all Rota-Baxter operators are integrable. As a byproduct, we construct a matched pair of Lie algebras that can not be integrated to a matched pair of Lie groups.
\begin{ex}{\rm
Consider the Lie algebra $\g=\so(3, \mathbb{R})$, whose integration is the connected and simply connected Lie group $\mathrm{SU}(2)$. Since $\mathrm{SU}(2)$ is compact, by Corollary \ref{Cor1}, then any Rota-Baxter operator $B:\so(3, \mathbb{R})\lon\so(3, \mathbb{R})$ can be integrated into a Rota-Baxter operator $\huaB: \mathrm{SU}(2)\lon \mathrm{SU}(2)$.
}
\end{ex}

\begin{ex}{\rm
Consider the Heisenberg Lie algebra $(\g, [\cdot,\cdot])$ with the basis
$$
X=
\left(
    \begin{array}{ccc}
      0 & 1 & 0 \\
      0 & 0 & 0 \\
      0 & 0 & 0 \\
    \end{array}
  \right), \quad
Y=
\left(
  \begin{array}{ccc}
    0 & 0 & 0 \\
    0 & 0 & 1 \\
    0 & 0 & 0 \\
  \end{array}
\right), \quad
Z=
\left(
  \begin{array}{ccc}
    0 & 0 & 1 \\
    0 & 0 & 0 \\
    0 & 0 & 0 \\
  \end{array}
\right),
$$
and
$$
[X, Y]=Z, \quad [X, Z]=0, \quad [Y, Z]=0.
$$
Then $G=\{\left(
            \begin{array}{ccc}
              1 & a & b \\
              0 & 1 & c \\
              0 & 0 & 1 \\
            \end{array}
          \right)
| a, b, c\in\mathbb{R}
\}$ is the connected and simply connected Lie group integrating the Heisenberg Lie algebra $\g$.
Define a linear map $B:\g\lon\g$ by
$$
B(X, Y, Z)=(X, Y, Z)\left(
                      \begin{array}{ccc}
                        0 & 0 & 0 \\
                        0 & 0 & 0 \\
                        0 & 1 & 0 \\
                      \end{array}
                    \right).
$$
By straightforward computation, $B$ is a Rota-Baxter operator on $(\g, [\cdot, \cdot])$, and the Lie subalgebra $Gr(B)\subseteq\g\ltimes_\ad\g$ is given by
\begin{eqnarray*}
&&Gr(B)\\
&=&\text{span}\{\Big(\left(
                          \begin{array}{ccc}
                            0 & 0 & 0 \\
                            0 & 0 & 0 \\
                            0 & 0 & 0 \\
                          \end{array}
                        \right),
                        \left(
                          \begin{array}{ccc}
                            0 & 1 & 0 \\
                            0 & 0 & 0 \\
                            0 & 0 & 0 \\
                          \end{array}
                        \right)
\Big),\Big(\left(
                          \begin{array}{ccc}
                            0 & 0 & 1 \\
                            0 & 0 & 0 \\
                            0 & 0 & 0 \\
                          \end{array}
                        \right),
                        \left(
                          \begin{array}{ccc}
                            0 & 0 & 0 \\
                            0 & 0 & 1 \\
                            0 & 0 & 0 \\
                          \end{array}
                        \right)
\Big),\Big(\left(
                          \begin{array}{ccc}
                            0 & 0 & 0 \\
                            0 & 0 & 0 \\
                            0 & 0 & 0 \\
                          \end{array}
                        \right),
                        \left(
                          \begin{array}{ccc}
                            0 & 0 & 1 \\
                            0 & 0 & 0 \\
                            0 & 0 & 0 \\
                          \end{array}
                        \right)
\Big)\}.
\end{eqnarray*}
Denote by
$$
D=\{\left(
    \begin{array}{ccc}
      1 & 0 & c \\
      0 & 1 & 0 \\
      0 & 0 & 1 \\
    \end{array}
  \right), \left(
             \begin{array}{ccc}
               1 & a & b \\
               0 & 1 & c \\
               0 & 0 & 1 \\
             \end{array}
           \right)|a, b, c\in\mathbb{R}\}.
$$
Then $D\subseteq G\ltimes_\Ad G$ is the closed connected Lie subgroup and $T_{(e_G, e_G)}D=Gr(B)$. Moreover, by straightforward computation, it follows that $D\cdot_\Ad (G\times e_G)=G\ltimes_\Ad G$. By Theorem \ref{proint2}, the Rota-Baxter $B:\g\lon\g$ can be integrated into a Rota-Baxter operator $\huaB: G\lon G$. By   \eqref{eqrbnew}, the Rota-Baxter operator $\huaB$ is given by
\begin{equation*}
\huaB(\left(
        \begin{array}{ccc}
          1 & a & b \\
          0 & 1 & c \\
          0 & 0 & 1 \\
        \end{array}
      \right)
)=\left(
          \begin{array}{ccc}
            1 & 0 & c \\
            0 & 1 & 0 \\
            0 & 0 & 1 \\
          \end{array}
        \right).
\end{equation*}
}
\end{ex}

In fact, any Rota-Baxter operator on the Heisenberg Lie algebra is integrable. See \cite{Rat} for more details.

\begin{ex}\label{exint}{\rm
Consider the $2$-dimensional real Lie algebra
$$\Big(\g=\text{span}\{e_1=\left(
\begin{array}{cc}
1 & 0 \\
0 & -1 \\
\end{array}
\right), e_2=\left(
\begin{array}{cc}
0 & 1 \\
0 & 0 \\
\end{array}
\right)
\}, ~~~~[\cdot, \cdot]\Big),$$ where $$[e_1, e_2]=e_1e_2-e_2e_1=2e_2.$$
Then $G=\{\left(                                                                                                                             \begin{array}{cc}                                                                                                                                 a & b \\                                                                                                                                 0 & \frac{1}{a} \\                                                                                                                               \end{array}                                                                                                                             \right)| a>0, b\in\mathbb{R}
\}$ is the connected and simply connected Lie group integrating the Lie algebra $\g$.
A linear map $B:\g\lon\g$ given by
$$
B(e_1, e_2)=(e_1, e_2)
\left(
\begin{array}{cc}
b_{11} & b_{12} \\
b_{21} & b_{22} \\
\end{array}
\right)
$$
is a Rota-Baxter operator if and only if
$$
[B(e_1), B(e_2)]=B([B(e_1), e_2]+[e_1, B(e_2)]+[e_1, e_2]).
$$
By straightforward computation, there are the following three classes of Rota-Baxter operators:
$$
B=
\left(
\begin{array}{cc}
\lambda & 0 \\
\mu & 0 \\
\end{array}
\right), \quad \text{or} \quad
B=
\left(
\begin{array}{cc}
\lambda & 0 \\
\mu & -1 \\
\end{array}
\right),
\quad \text{or} \quad B=
\left(
\begin{array}{cc}
-m-1 & k \\
\frac{-m^2-m}{k} & m \\
\end{array}
\right), \quad k\neq0.
$$
In the following, we will analyse the integrability of above Rota-Baxter operators on $\g$.

{\bf Case $1.1:$} the Rota-Baxter operator $B:\g\lon\g$ is given by
 $$
B(e_1, e_2)=(e_1, e_2)
\left(
\begin{array}{cc}
\lambda & 0 \\
\mu & 0 \\
\end{array}
\right), \quad \text{where}\quad\lambda\neq-1, \lambda\neq0,\mu\in\mathbb{R}.
 $$
The Lie algebra $Gr(B)\subseteq\g\ltimes_\ad\g$ associated to $B$ is
$$
Gr(B)=\text{span}\{\Big(\left(                                                                                                                                   \begin{array}{cc}                                                                                                                                     \lambda & \mu \\                                                                                                                                     0 & -\lambda \\                                                                                                                                   \end{array}                                                                                                                                 \right), \left(                                                                                                                                            \begin{array}{cc}                                                                                                                                              1 & 0 \\                                                                                                                                              0 & -1 \\                                                                                                                                            \end{array}                                                                                                                                          \right)\Big), \Big(\left(                                                                                                                                                               \begin{array}{cc}                                                                                                                                                                 0 & 0 \\                                                                                                                                                                 0 & 0 \\                                                                                                                                                               \end{array}                                                                                                                                                             \right),\left(                                                                                                                                                                       \begin{array}{cc}                                                                                                                                                                         0 & 1 \\                                                                                                                                                                         0 & 0 \\                                                                                                                                                                       \end{array}                                                                                                                                                                     \right)                                                                                                                                          \Big)
\}.$$
Denote by
$$
D=\{\left(
\begin{array}{cc}
a^{\lambda} & \frac{\mu}{2\lambda}(a^{\lambda}-a^{-\lambda}) \\
0 & a^{-\lambda} \\
\end{array}
\right),
\left(
\begin{array}{cc}
a & -a^{(\lambda+1)}\frac{\mu}{2\lambda}(a^{\lambda}-a^{-\lambda})+a^{\lambda}\frac{\mu}{2(\lambda+1)}(a^{(\lambda+1)}-a^{-(\lambda+1)})+sa^{-1} \\
0 & a^{-1} \\
\end{array}
\right)|a>0, s\in\mathbb{R}
\}.
$$
Then $D\subseteq G\ltimes_\Ad G$ is the closed connected Lie subgroup and $T_{(e_G, e_G)}D=Gr(B)$. Moreover,
\begin{eqnarray*}
D\cdot_\Ad (G\times e_G)&\subseteq&\emptycomment{\{
\left(
\begin{array}{cc}
a^{\lambda}p & a^{\lambda}q+\frac{\mu}{2\lambda}(a^{\lambda}-a^{-\lambda})p^{-1} \\
0 & a^{-\lambda}p^{-1} \\
\end{array}
\right),
\left(
\begin{array}{cc}
a & -a^{(\lambda+1)}\frac{\mu}{2\lambda}(a^{\lambda}-a^{-\lambda})+a^{\lambda }\frac{\mu}{2(\lambda+1)}(a^{(\lambda+1)}-a^{-(\lambda+1)})+sa^{-1} \\
0 & a^{-1} \\
\end{array}
\right)
\}.\\}
 G\ltimes_\Ad G.
\end{eqnarray*}
For any
$\Big(
\left(
  \begin{array}{cc}
    x & y \\
    0 & \frac{1}{x} \\
  \end{array}
\right),
\left(
  \begin{array}{cc}
    z & w \\
    0 & \frac{1}{z} \\
  \end{array}
\right)
\Big)\in G\ltimes_\Ad G$, let
\begin{eqnarray*}
\left\{\begin{array}{rcl}
p&=&\frac{x}{z^{\lambda}}, \\
q&=&\frac{y}{z^{\lambda}}-\frac{\mu}{2\lambda x}(z^{\lambda}-z^{-\lambda}),\\
a&=&z, \\
s&=&zw+z^{(\lambda+2)}\frac{\mu}{2\lambda}(z^{\lambda}-z^{-\lambda})-z^{\lambda+1 }\frac{\mu}{2(\lambda+1)}(z^{(\lambda+1)}-z^{-(\lambda+1)}).
\end{array}\right.
\end{eqnarray*}
Then
\begin{eqnarray*}
&&
\Big(
\left(
  \begin{array}{cc}
    x & y \\
    0 & \frac{1}{x} \\
  \end{array}
\right),
\left(
  \begin{array}{cc}
    z & w \\
    0 & \frac{1}{z} \\
  \end{array}
\right)
\Big)\\
&=&\Big(
\left(
\begin{array}{cc}a^{\lambda} & \frac{\mu}{2\lambda}(a^{\lambda}-a^{-\lambda}) \\
0 & a^{-\lambda} \\
\end{array}
\right),
\left(
\begin{array}{cc}
a & -a^{(\lambda+1)}\frac{\mu}{2\lambda}(a^{\lambda}-a^{-\lambda})+a^{\lambda}\frac{\mu}{2(\lambda+1)}(a^{(\lambda+1)}-a^{-(\lambda+1)})+sa^{-1} \\
0 & a^{-1} \\
\end{array}
\right)\Big)\\
&&\cdot_\Ad \Big(
\left(
\begin{array}{cc}p & q \\
0 & p^{-1} \\
\end{array}
\right),
\left(
\begin{array}{cc}
 1& 0 \\
0 & 1 \\
\end{array}
\right)
\Big)\in D\cdot_\Ad (G\times e_G).
\end{eqnarray*}
Thus $D\cdot_\Ad (G\times e_G)=G\ltimes_\Ad G$. By Theorem \ref{proint2}, the Rota-Baxter operator $B:\g\lon\g$ can be integrated into a Rota-Baxter operator $\huaB: G\lon G$.

In fact the smooth map $\huaB: G\lon G$ given by
$$
\huaB
\left(
\begin{array}{cc}
a & b \\
 0 & a^{-1} \\
\end{array}
\right)=
\left(
\begin{array}{cc}
a^{\lambda} & \frac{\mu}{2\lambda}(a^{\lambda}-a^{-\lambda}) \\
 0 & a^{-\lambda} \\
\end{array}
\right),
$$
is a Rota-Baxter operator on the Lie group $G$ integrating
$B=
\left(
\begin{array}{cc}
\lambda & 0 \\
\mu & 0 \\
\end{array}
\right)
$.

{\bf Case $1.2:$} the Rota-Baxter operator $B:\g\lon\g$ is given by
$$
B(e_1, e_2)=(e_1, e_2)
\left(
  \begin{array}{cc}
    -1 & 0 \\
    \mu & 0 \\
  \end{array}
\right),
\quad \text{where}\quad\mu\in\mathbb{R}.
$$
The Lie algebra $Gr(B)\subseteq\g\ltimes_\ad\g$ associated to $B$ is
$$
Gr(B)=\text{span}\{\Big(\left(                                                                                                                                   \begin{array}{cc}                                                                                                                                     -1& \mu \\                                                                                                                                     0 & 1 \\                                                                                                                                   \end{array}                                                                                                                                 \right), \left(                                                                                                                                            \begin{array}{cc}                                                                                                                                              1 & 0 \\                                                                                                                                              0 & -1 \\                                                                                                                                            \end{array}                                                                                                                                          \right)\Big), \Big(\left(                                                                                                                                                               \begin{array}{cc}                                                                                                                                                                 0 & 0 \\                                                                                                                                                                 0 & 0 \\                                                                                                                                                               \end{array}                                                                                                                                                             \right),\left(                                                                                                                                                                       \begin{array}{cc}                                                                                                                                                                         0 & 1 \\                                                                                                                                                                         0 & 0 \\                                                                                                                                                                       \end{array}                                                                                                                                                                     \right)                                                                                                                                          \Big)
\}.$$
Denote by
$$
D=\{\left(
\begin{array}{cc}
a & \frac{-\mu}{2}(a-a^{-1}) \\
0 & a^{-1} \\
\end{array}
\right),
\left(
\begin{array}{cc}
a^{-1} & \frac{\mu}{2}(a-a^{-1})-\mu a\ln a+sa \\
0 & a \\
\end{array}
\right)|a>0, s\in\mathbb{R}
\}.
$$
Then $D\subseteq G\ltimes_\Ad G$ is the closed connected Lie subgroup and $T_{(e_G, e_G)}D=Gr(B)$. Moreover, by similar computation in case $1.1$,
\emptycomment{
\begin{eqnarray*}
D\cdot_\Ad G=\{
\left(
\begin{array}{cc}
ap & aq-\frac{\mu}{2}(a-a^{-1})p^{-1} \\
0 & a^{-1}p^{-1} \\
\end{array}
\right),
\left(
\begin{array}{cc}
a^{-1} & \frac{\mu}{2}(a-a^{-1})-\mu alna+sa \\
0 & a \\
\end{array}
\right)|a, p>0, s, q\in\mathbb{R}
\}.
\end{eqnarray*}
For any
$\Big(
\left(
  \begin{array}{cc}
    x & y \\
    0 & \frac{1}{x} \\
  \end{array}
\right),
\left(
  \begin{array}{cc}
    z & w \\
    0 & \frac{1}{z} \\
  \end{array}
\right)
\Big)\in G\cdot_{\ltimes} H$, let
\begin{eqnarray*}
\left\{\begin{array}{rcl}
p&=&xz, \\
s&=&zw-\frac{\mu}{2}(1-z^2)-\mu lnz,\\
a&=&\frac{1}{z}, \\
q&=&yz+\frac{\mu}{2x}(\frac{1}{z}-z),
\end{array}\right.
\end{eqnarray*}
we have that
\begin{eqnarray*}
&&\Big(
\left(
\begin{array}{cc}
ap & aq-\frac{\mu}{2}(a-a^{-1})p^{-1} \\
0 & a^{-1}p^{-1} \\
\end{array}
\right),
\left(
\begin{array}{cc}
a^{-1} & \frac{\mu}{2}(a-a^{-1})-\mu alna+sa \\
0 & a \\
\end{array}
\right)\Big)\\
&=&\Big(
\left(
  \begin{array}{cc}
    x & y \\
    0 & \frac{1}{x} \\
  \end{array}
\right),
\left(
  \begin{array}{cc}
    z & w \\
    0 & \frac{1}{z} \\
  \end{array}
\right)
\Big)\in D\cdot_\Ad G.
\end{eqnarray*}
}
it follows that $D\cdot_\Ad (G\times e_G)=G\ltimes_\Ad G$. By Theorem \ref{proint2}, the Rota-Baxter operator $B:\g\lon\g$ can be integrated into a Rota-Baxter operator $\huaB: G\lon G$.

In fact the smooth map $\huaB:G\lon G$ defined by
$$
\huaB
\left(
\begin{array}{cc}
a & b \\
 0 & a^{-1} \\
\end{array}
\right)=
\left(
\begin{array}{cc}
a^{-1} & -\frac{\mu}{2}(a^{-1}-a) \\
 0 & a \\
\end{array}
\right),
$$
is a Rota-Baxter operator on the Lie group $G$
integrating
$B=
\left(
\begin{array}{cc}
-1 & 0 \\
\mu & 0 \\
\end{array}
\right)
$.

{\bf Case $1.3:$} the Rota-Baxter operator $B:\g\lon\g$ is given by
$$
B(e_1, e_2)=(e_1, e_2)
\left(
  \begin{array}{cc}
    0 & 0 \\
    \mu & 0 \\
  \end{array}
\right),
 \quad \text{where}\quad\mu\in\mathbb{R}.
$$
The Lie algebra $Gr(B)\subseteq\g\ltimes_\ad\g$ associated to $B$ is
$$
Gr(B)=\text{span}\{\Big(\left(                                                                                                                                   \begin{array}{cc}                                                                                                                                     0 & \mu \\                                                                                                                                     0 & 0\\                                                                                                                                   \end{array}                                                                                                                                 \right), \left(                                                                                                                                            \begin{array}{cc}                                                                                                                                              1 & 0 \\                                                                                                                                              0 & -1 \\                                                                                                                                            \end{array}                                                                                                                                          \right)\Big), \Big(\left(                                                                                                                                                               \begin{array}{cc}                                                                                                                                                                 0 & 0 \\                                                                                                                                                                 0 & 0 \\                                                                                                                                                               \end{array}                                                                                                                                                             \right),\left(                                                                                                                                                                       \begin{array}{cc}                                                                                                                                                                         0 & 1 \\                                                                                                                                                                         0 & 0 \\                                                                                                                                                                       \end{array}                                                                                                                                                                     \right)                                                                                                                                          \Big)
\}.$$
Denote by
$$
D=\{\left(
\begin{array}{cc}
1 & \mu \ln a\\
0 & 1 \\
\end{array}
\right),
\left(
\begin{array}{cc}
a & \frac{\mu}{2}(a-a^{-1})-\mu a\ln a+sa^{-1} \\
0 & a^{-1} \\
\end{array}
\right)|a>0, s\in\mathbb{R}
\}.
$$
Then $D\subseteq G\ltimes_\Ad G$ is the closed connected Lie subgroup and $T_{(e_G, e_G)}D=Gr(B)$. Moreover,
\emptycomment{
\begin{eqnarray*}
D\cdot_\Ad G=\{
\left(
\begin{array}{cc}
p & q+\frac{a\mu}{p} \\
0 & p^{-1} \\
\end{array}
\right),
\left(
\begin{array}{cc}
e^a & \frac{\mu}{2}(e^{a}-e^{-a})-\mu ae^{a}+se^{-a} \\
0 & e^{-a} \\
\end{array}
\right)
\}.
\end{eqnarray*}
For any
$\Big(
\left(
  \begin{array}{cc}
    x & y \\
    0 & \frac{1}{x} \\
  \end{array}
\right),
\left(
  \begin{array}{cc}
    z & w \\
    0 & \frac{1}{z} \\
  \end{array}
\right)
\Big)\in G\cdot_{\ltimes} H$, let
\begin{eqnarray*}
\left\{\begin{array}{rcl}
a&=&lnz, \\
s&=&zw-\frac{\mu}{2}(z^2-1)+\mu z^{2}lnz,\\
p&=&x, \\
q&=&y-\frac{\mu lnz}{x},
\end{array}\right.
\end{eqnarray*}
we have that
\begin{eqnarray*}
&&\Big(\left(
\begin{array}{cc}
p & q+\frac{a\mu}{p} \\
0 & p^{-1} \\
\end{array}
\right),
\left(
\begin{array}{cc}
e^a & \frac{\mu}{2}(e^{a}-e^{-a})-\mu ae^{a}+se^{-a} \\
0 & e^{-a} \\
\end{array}
\right)\Big)\\
&=&\Big(
\left(
  \begin{array}{cc}
    x & y \\
    0 & \frac{1}{x} \\
  \end{array}
\right),
\left(
  \begin{array}{cc}
    z & w \\
    0 & \frac{1}{z} \\
  \end{array}
\right)
\Big)\in D\cdot_\Ad G.
\end{eqnarray*}
}
by similar computation in case $1.1$, it follows that $D\cdot_\Ad (G\times e_G)=G\ltimes_\Ad G$. By Theorem \ref{proint2}, the Rota-Baxter operator $B:\g\lon\g$ can be integrated into a Rota-Baxter operator $\huaB: G\lon G$.

In fact the smooth map $\huaB:G\lon G$ defined by
$$
\huaB
\left(
\begin{array}{cc}
a & b \\
 0 & a^{-1} \\
\end{array}
\right)=
\left(
\begin{array}{cc}
1 & \mu \ln a \\
 0 & 1 \\
\end{array}
\right),
$$
is a Rota-Baxter operator on the Lie group $G$
integrating
$B=
\left(
\begin{array}{cc}
0 & 0 \\
\mu & 0 \\
\end{array}
\right).
$

{\bf Case $2.1:$} the Rota-Baxter operator $B:\g\lon\g$ is given by
$$
B(e_1, e_2)=(e_1, e_2)
\left(
  \begin{array}{cc}
    0 & 0 \\
    \mu & -1 \\
  \end{array}
\right), \quad \text{where}\quad \mu\in\mathbb{R}.
$$
The Lie algebra $Gr(B)\subseteq\g\ltimes_\ad\g$ associated to $B$ is
$$
Gr(B)=\text{span}\{\Big(\left(                                                                                                                                   \begin{array}{cc}                                                                                                                                     0 & \mu\\                                                                                                                                     0 & 0\\                                                                                                                                   \end{array}                                                                                                                                 \right), \left(                                                                                                                                            \begin{array}{cc}                                                                                                                                              1 & 0 \\                                                                                                                                              0 & -1 \\                                                                                                                                            \end{array}                                                                                                                                          \right)\Big), \Big(\left(                                                                                                                                                               \begin{array}{cc}                                                                                                                                                                 0 & -1 \\                                                                                                                                                                 0 & 0 \\                                                                                                                                                               \end{array}                                                                                                                                                             \right),\left(                                                                                                                                                                       \begin{array}{cc}                                                                                                                                                                         0 & 1 \\                                                                                                                                                                         0 & 0 \\                                                                                                                                                                       \end{array}                                                                                                                                                                     \right)                                                                                                                                          \Big)
\}.$$
Denote by
$$
D=\{\left(
\begin{array}{cc}
1 & -s+\mu\ln a\\
0 & 1 \\
\end{array}
\right),
\left(
\begin{array}{cc}
a & sa+\frac{\mu}{2}(a-a^{-1})-\mu a\ln a\\
0 & a^{-1} \\
\end{array}
\right)|a>0, s\in\mathbb{R}
\}.
$$
Then $D\subseteq G\ltimes_\Ad G$ is the closed connected Lie subgroup and $T_{(e_G, e_G)}=Gr(B)$. Moreover,
\emptycomment{
\begin{eqnarray*}
D\cdot_\Ad G=\{
\left(
\begin{array}{cc}
p & q+\frac{\mu a-s}{p} \\
0 & p^{-1} \\
\end{array}
\right),
\left(
\begin{array}{cc}
e^a & se^a+\frac{\mu}{2}(e^a-e^{-a})-\mu ae^{a} \\
0 & e^{-a} \\
\end{array}
\right)
\}.
\end{eqnarray*}
For any
$\Big(
\left(
  \begin{array}{cc}
    x & y \\
    0 & \frac{1}{x} \\
  \end{array}
\right),
\left(
  \begin{array}{cc}
    z & w \\
    0 & \frac{1}{z} \\
  \end{array}
\right)
\Big)\in G\cdot_{\ltimes} H$, let
\begin{eqnarray*}
\left\{\begin{array}{rcl}
a&=&lnz, \\
s&=&\mu lnz+\frac{w}{z}+\frac{\mu}{2}(1-\frac{1}{z^2}),\\
p&=&x, \\
q&=&y+\frac{1}{x}(\frac{w}{z}+\frac{\mu}{2}(1-\frac{1}{z^2})),
\end{array}\right.
\end{eqnarray*}
we have that
\begin{eqnarray*}
&&\Big(\left(
\begin{array}{cc}
p & q+\frac{\mu a-s}{p} \\
0 & p^{-1} \\
\end{array}
\right),
\left(
\begin{array}{cc}
e^a & se^a+\frac{\mu}{2}(e^a-e^{-a})-\mu ae^{a} \\
0 & e^{-a} \\
\end{array}
\right)\Big)=\Big(
\left(
  \begin{array}{cc}
    x & y \\
    0 & \frac{1}{x} \\
  \end{array}
\right),
\left(
  \begin{array}{cc}
    z & w \\
    0 & \frac{1}{z} \\
  \end{array}
\right)
\Big)\in D\cdot_\Ad G.
\end{eqnarray*}
}
by similar computation in case $1.1$, it follows that $D\cdot_\Ad (G\times e_G)=G\ltimes_\Ad G$. By Theorem \ref{proint2}, the Rota-Baxter operator $B:\g\lon\g$ can be integrated into a Rota-Baxter operator $\huaB: G\lon G$.

In fact the smooth map $\huaB:G\lon G$ defined by
$$
\huaB
\left(
\begin{array}{cc}
a & b \\
 0 & a^{-1} \\
\end{array}
\right)=
\left(
\begin{array}{cc}
1 & -\frac{b}{a}+\frac{\mu}{2}(1-a^{-2}) \\
 0 & 1 \\
\end{array}
\right),
$$
is a Rota-Baxter operator on the Lie group $G$
integrating
$B=
\left(
\begin{array}{cc}
0 & 0 \\
\mu & -1 \\
\end{array}
\right).
$

{\bf Case $2.2:$} the Rota-Baxter operator $B:\g\lon\g$ is given by
$$
B(e_1, e_2)=(e_1, e_2)
\left(
  \begin{array}{cc}
    -1 & 0 \\
    \mu & -1 \\
  \end{array}
\right), \quad \text{where}\quad  \mu\in\mathbb{R}.
$$
The Lie algebra $Gr(B)\subseteq\g\ltimes_\ad\g$ associated to $B$ is
$$
Gr(B)=\text{span}\{\Big(\left(                                                                                                                                   \begin{array}{cc}                                                                                                                                     -1& \mu \\                                                                                                                                     0 & 1\\                                                                                                                                   \end{array}                                                                                                                                 \right), \left(                                                                                                                                            \begin{array}{cc}                                                                                                                                              1 & 0 \\                                                                                                                                              0 & -1 \\                                                                                                                                            \end{array}                                                                                                                                          \right)\Big), \Big(\left(                                                                                                                                                               \begin{array}{cc}                                                                                                                                                                 0 & -1 \\                                                                                                                                                                 0 & 0 \\                                                                                                                                                               \end{array}                                                                                                                                                             \right),\left(                                                                                                                                                                       \begin{array}{cc}                                                                                                                                                                         0 & 1 \\                                                                                                                                                                         0 & 0 \\                                                                                                                                                                       \end{array}                                                                                                                                                                     \right)                                                                                                                                          \Big)
\}.$$
Denote by
$$
D=\{\left(
\begin{array}{cc}
a & \frac{-\mu}{2}(a-a^{-1})-sa^{-1} \\
0 & a^{-1} \\
\end{array}
\right),
\left(
\begin{array}{cc}
a^{-1} & \frac{\mu}{2}(a-a^{-1})-\mu a\ln a+sa^{-1} \\
0 & a \\
\end{array}
\right)|a>0, s\in\mathbb{R}
\}.
$$
Then $D\subseteq G\ltimes_\Ad G$ is the closed connected Lie subgroup and $T_{(e_G, e_G)}D=Gr(B)$. Moreover,
\emptycomment{
\begin{eqnarray*}
D\cdot_\Ad G=\{
\left(
\begin{array}{cc}
ap & aq+\frac{1}{p}(-\frac{\mu}{2}(a-a^{-1})-sa^{-1}) \\
0 & a^{-1}p^{-1} \\
\end{array}
\right),
\left(
\begin{array}{cc}
a^{-1} & \frac{\mu}{2}(a-a^{-1})-\mu alna+sa^{-1} \\
0 & a \\
\end{array}
\right)
\}.
\end{eqnarray*}
For any
$\Big(
\left(
  \begin{array}{cc}
    x & y \\
    0 & \frac{1}{x} \\
  \end{array}
\right),
\left(
  \begin{array}{cc}
    z & w \\
    0 & \frac{1}{z} \\
  \end{array}
\right)
\Big)\in G\cdot_{\ltimes} H$, let
\begin{eqnarray*}
\left\{\begin{array}{rcl}
p&=&xz, \\
s&=&\frac{1}{z}\Big(w-\frac{\mu}{2}(\frac{1}{z}-z)-\mu \frac{lnz}{z}\Big),\\
a&=&\frac{1}{z}, \\
q&=&zy-\frac{1}{x}\Big(\mu \frac{lnz}{z}-w\Big),
\end{array}\right.
\end{eqnarray*}
we have that
\begin{eqnarray*}
&&\Big(\left(
\begin{array}{cc}
ap & aq+\frac{1}{p}(-\frac{\mu}{2}(a-a^{-1})-sa^{-1}) \\
0 & a^{-1}p^{-1} \\
\end{array}
\right),
\left(
\begin{array}{cc}
a^{-1} & \frac{\mu}{2}(a-a^{-1})-\mu alna+sa^{-1} \\
0 & a \\
\end{array}
\right)\Big)\\
&=&\Big(
\left(
  \begin{array}{cc}
    x & y \\
    0 & \frac{1}{x} \\
  \end{array}
\right),
\left(
  \begin{array}{cc}
    z & w \\
    0 & \frac{1}{z} \\
  \end{array}
\right)
\Big)\in D\cdot_\Ad G.
\end{eqnarray*}
}
by similar computation in case $1.1$, it follows that $D\cdot_\Ad (G\times e_G)=G\ltimes_\Ad G$. By Theorem \ref{proint2}, the Rota-Baxter operator $B:\g\lon\g$ can be integrated into a Rota-Baxter operator $\huaB: G\lon G$.

In fact the smooth map $\huaB:G\lon G$ defined by
$$
\huaB
\left(
\begin{array}{cc}
a & b \\
 0 & a^{-1} \\
\end{array}
\right)=
\left(
\begin{array}{cc}
a^{-1} & -b+\frac{\mu \ln a}{a} \\
 0 & a \\
\end{array}
\right),
$$
is a Rota-Baxter operator on the Lie group $G$
integrating
$B=
\left(
\begin{array}{cc}
-1 & 0 \\
\mu & -1 \\
\end{array}
\right).
$

{\bf Case $2.3:$} the Rota-Baxter operator $B:\g\lon\g$ is given by
$$
B(e_1, e_2)=(e_1, e_2)
\left(
  \begin{array}{cc}
    \lambda & 0 \\
    \mu & -1 \\
  \end{array}
\right),
\quad \text{where}\quad  \lambda\neq0, -1, \mu\in\mathbb{R}.
$$
The Lie algebra $Gr(B)\subseteq\g\ltimes_\ad\g$ associated to $B$ is
$$
Gr(B)=\text{span}\{\Big(\left(                                                                                                                                   \begin{array}{cc}                                                                                                                                     \lambda& \mu \\                                                                                                                                     0 & -\lambda\\                                                                                                                                   \end{array}                                                                                                                                 \right), \left(                                                                                                                                            \begin{array}{cc}                                                                                                                                              1 & 0 \\                                                                                                                                              0 & -1 \\                                                                                                                                            \end{array}                                                                                                                                          \right)\Big), \Big(\left(                                                                                                                                                               \begin{array}{cc}                                                                                                                                                                 0 & -1 \\                                                                                                                                                                 0 & 0 \\                                                                                                                                                               \end{array}                                                                                                                                                             \right),\left(                                                                                                                                                                       \begin{array}{cc}                                                                                                                                                                         0 & 1 \\                                                                                                                                                                         0 & 0 \\                                                                                                                                                                       \end{array}                                                                                                                                                                     \right)                                                                                                                                          \Big)
\}.$$
Denote by
\begin{eqnarray*}
D&=&\{\left(
\begin{array}{cc}
a^{\lambda} & \frac{\mu}{2\lambda}(a^{\lambda}-a^{-\lambda})-sa^{-\lambda} \\
0 & a^{-\lambda} \\
\end{array}
\right),\\
&&
\left(
\begin{array}{cc}
a & -\frac{\mu}{2\lambda}a^{(\lambda+1)}(a^{\lambda}-a^{-\lambda})+\frac{\mu}{2(\lambda+1)}(a^{(\lambda+1)}-a^{-(\lambda+1)}) a^{\lambda}+sa \\
0 & a^{-1} \\
\end{array}
\right)|a>0, s\in\mathbb{R}
\}.
\end{eqnarray*}
Then $D\subseteq G\ltimes_\Ad G$ is the closed connected Lie subgroup and $T_{(e_G, e_G)}D=Gr(B)$. Moreover,
by similar computation in case $1.1$, it follows that $D\cdot_\Ad (G\times e_G)=G\ltimes_\Ad G$. By Theorem \ref{proint2},  the Rota-Baxter operator $B:\g\lon\g$ can be integrated into a Rota-Baxter operator $\huaB: G\lon G$.

In fact the smooth map $\huaB: G\lon G$ defined by
$$
\huaB
\left(
\begin{array}{cc}
a & b \\
 0 & a^{-1} \\
\end{array}
\right)=
\left(
\begin{array}{cc}
a^{\lambda} & \frac{\mu a^{\lambda}}{2(\lambda+1)}-\frac{\mu a^{-\lambda-2}}{2(\lambda+1)}-\frac{b}{a^{\lambda+1}} \\
 0 & a^{-\lambda} \\
\end{array}
\right),
$$
is a Rota-Baxter operator on the Lie group $G$
integrating
$B=
\left(
\begin{array}{cc}
\lambda & 0 \\
\mu & -1 \\
\end{array}
\right).
$

{\bf Case $3:$} the Rota-Baxter operator $B:\g\lon\g$ is given by
\begin{eqnarray*}
B(e_1, e_2)=(e_1, e_2)
\left(
  \begin{array}{cc}
    -(m+1) & k \\
    -\frac{m+m^2}{k} & m \\
  \end{array}
\right),
 \quad \text{where}\quad k\neq 0.
\end{eqnarray*}
The Lie algebra $Gr(B)\subseteq\g\ltimes_\ad\g$ associated to $B$ is
$$
Gr(B)=\text{span}\{\Big(\left(                                                                                                                                   \begin{array}{cc}                                                                                                                                     -m-1& -\frac{m+m^2}{k} \\                                                                                                                                     0 & m+1\\                                                                                                                                   \end{array}                                                                                                                                 \right), \left(                                                                                                                                            \begin{array}{cc}                                                                                                                                              1 & 0 \\                                                                                                                                              0 & -1 \\                                                                                                                                            \end{array}                                                                                                                                          \right)\Big), \Big(\left(                                                                                                                                                               \begin{array}{cc}                                                                                                                                                                 k & m \\                                                                                                                                                                 0 & -k \\                                                                                                                                                               \end{array}                                                                                                                                                             \right),\left(                                                                                                                                                                       \begin{array}{cc}                                                                                                                                                                         0 & 1 \\                                                                                                                                                                         0 & 0 \\                                                                                                                                                                       \end{array}                                                                                                                                                                     \right)                                                                                                                                          \Big)
\}.$$
Denote by
\begin{eqnarray*}
D&=&\{\left(
\begin{array}{cc}
a & \frac{m}{2k}(a-a^{-1})\\
0 & a^{-1} \\
\end{array}
\right),\\
&&
\left(
\begin{array}{cc}
b & \frac{ab^{-m}}{2k}(ab^{m+1}-a^{-1}b^{-m-1})+\frac{m}{2k}(b-b^{-1})+\frac{1}{2k}(b^{(-2m-1)}-b^{-1})\\
0 & b^{-1} \\
\end{array}
\right)|a, b>0\}.
\end{eqnarray*}
Then $D\subseteq G\ltimes_\Ad G$ is the closed connected Lie subgroup and $T_{(e_G, e_G)}D=Gr(B)$. \emptycomment{Moreover,
\begin{eqnarray*}
D\cdot_\Ad (G\times e_G)&\subseteq&\emptycomment{\{
\left(
\begin{array}{cc}
ap & aq+\frac{m}{2kp}(a-a^{-1}) \\
0 & a^{-1}p^{-1} \\
\end{array}
\right),
\left(
\begin{array}{cc}
b & \frac{ab^{-m}}{2k}(ab^{m+1}-a^{-1}b^{-m-1})+\frac{m}{2k}(b-b^{-1})+\frac{1}{2k}(b^{(-2m-1)}-b^{-1})\\
0 & b^{-1} \\
\end{array}
\right)
\}
}
 G\cdot_\Ad G.
\end{eqnarray*}
}
If $D\cdot_\Ad (G\times e_G)=G\ltimes_{\Ad}G$, then for the $\Big(
\left(
  \begin{array}{cc}
    x & y \\
    0 & \frac{1}{x} \\
  \end{array}
\right),
\left(
  \begin{array}{cc}
    z & \frac{(m-1)z}{2k}-\frac{(m+1)}{2kz} \\
    0 & \frac{1}{z} \\
  \end{array}
\right)
\Big)\in G\ltimes_{\Ad} G$, there exists $a>0, p>0, b>0, q\in\mathbb{R}$, such that
\begin{eqnarray*}
&&\Big(
\left(
  \begin{array}{cc}
    x & y \\
    0 & \frac{1}{x} \\
  \end{array}
\right),
\left(
  \begin{array}{cc}
    z & \frac{(m-1)z}{2k}-\frac{(m+1)}{2kz} \\
    0 & \frac{1}{z} \\
  \end{array}
\right)
\Big)\\
&=&(\left(
\begin{array}{cc}
a & \frac{m}{2k}(a-a^{-1})\\
0 & a^{-1} \\
\end{array}
\right),
\left(
\begin{array}{cc}
b & \frac{ab^{-m}}{2k}(ab^{m+1}-a^{-1}b^{-m-1})+\frac{m}{2k}(b-b^{-1})+\frac{1}{2k}(b^{(-2m-1)}-b^{-1})\\
0 & b^{-1} \\
\end{array}
\right))\\
&&\cdot_\Ad \Big(
\left(
\begin{array}{cc}p & q \\
0 & p^{-1} \\
\end{array}
\right),
\left(
\begin{array}{cc}
 1& 0 \\
0 & 1 \\
\end{array}
\right)
\Big)
\end{eqnarray*}
which means
\begin{eqnarray*}
\left\{\begin{array}{rcl}
~~ap&=&x, \\
~~y&=&aq+\frac{m}{2kp}(a-a^{-1}),\\
~~b&=&z, \\ ~~\frac{(m-1)z}{2k}-\frac{(m+1)}{2kz}&=&\frac{ab^{-m}}{2k}(ab^{m+1}-a^{-1}b^{-m-1})+\frac{m}{2k}(b-b^{-1})+\frac{1}{2k}(b^{(-2m-1)}-b^{-1}),
\end{array}\right.
\end{eqnarray*}
we have that $$b=z, \quad a^2=-1,$$ which implies that the above equation does not have a solution. Thus $D\cdot_\Ad (G\times e_G)\neq G\ltimes_\Ad G$. By Theorem \ref{proint2}, it follows that  the Rota-Baxter operator $B:\g\lon\g$ is not integrable.

{\bf Conclusion:}
\[
\begin{array}{|c|c|c|c|}
\hline
\rm{Rota\mbox{-}Baxter~~~~operators}&B=\left(
\begin{array}{cc}
\lambda & 0 \\
\mu & 0 \\
\end{array}
\right)&B=\left(
\begin{array}{cc}
\lambda & 0 \\
\mu & -1 \\
\end{array}
\right)&B=\left(
\begin{array}{cc}
-m-1 & k \\
\frac{-m^2-m}{k} & m \\
\end{array}
\right) , ~~~~k\neq 0\\
\hline
\rm{integrabel}& \rm{Yes} & \rm{Yes} & \rm{No} \\
\hline
\end{array}~~.
\]
}
\end{ex}

A large class of Rota-Baxter operators are given by projections. More precisely, let $\h$ be a Lie algebra and $\h_1, \h_2$ be Lie subalgebras of $\h$ such that $\h=\h_1\oplus\h_2$ as vector spaces. Denote by $P_1, P_2$ by projections from $\h$ to $\h_1$ and $\h_2$ respectively. Then $-P_1, -P_2$ are Rota-Baxter operators. These Rota-Baxter operators are not always integrable as the following example shows.

\begin{ex}{\rm
Consider the Lie algebra $(\h, [\cdot,\cdot]_\h)$ with a basis $\{u_1, u_2\}$, where $[u_1, u_2]_\h=2(u_2-u_1)$. Then the linear map
$$
B:\h\lon\h, \quad B(k_1u_1+k_2u_2)=-k_2u_2, \quad k_1, k_2\in\mathbb{R}, ~u_1, u_2\in\h,
$$
is a Rota-Baxter operator on $(\h, [\cdot,\cdot]_\h)$. Define a Lie algebra isomorphism $\varphi: \g\lon\h$ by
$$
\varphi(e_1, e_2)=(u_1, u_2)\left(
                              \begin{array}{cc}
                                1 & -1 \\
                                0 & 1 \\
                              \end{array}
                            \right),
$$
where $(\g=\text{span}\{
e_1, e_2\}, [\cdot, \cdot])$ is the Lie algebra given in Example \ref{exint}. We obtain that
$$
B':\g\lon\g, \quad B'(e_1, e_2)=\varphi^{-1}B\varphi(e_1, e_2)=(e_1, e_2)\left(
                                                                           \begin{array}{cc}
                                                                             0 & -1 \\
                                                                             0 & -1 \\
                                                                           \end{array}
                                                                         \right),
$$
is a Rota-Baxter operator on $(\g, [\cdot,\cdot])$. By Example \ref{exint}, we know that the Rota-Baxter operator $B'$ is not integrable, thus $B$ is not integrable.
}
\end{ex}

As a byproduct, we give a matched pair of Lie algebras that can not be integrated into a matched pair of Lie groups at the end of this section.

Consider the Lie algebra $$\Big(\g=\text{span}\{
e_1=\left(
      \begin{array}{cc}
        1 & 0 \\
        0 & -1 \\
      \end{array}
    \right),
e_2=\left(
      \begin{array}{cc}
        0 & 1 \\
        0 & 0 \\
      \end{array}
    \right)
\}, [\cdot, \cdot]\Big),$$
given in Example \ref{exint}. The connected and simply connected Lie group integrating $\g$ is
$$G=\{
\left(
\begin{array}{cc}
a & b \\
0 & \frac{1}{a} \\
\end{array}
\right)|a>0, b\in\mathbb{R}\}.
$$
It shows that the map
$$
B:\g\lon\g, \quad B(e_1, e_2)=(e_1, e_2)
\left(
\begin{array}{cc}
-m-1 & k \\
\frac{-m^2-m}{k} & m \\
\end{array}
\right),
$$
is a Rota-Baxter operator in the Example \ref{exint}. By Proposition \ref{matched}, we have that $(\g, Gr(B); \bar{\phi}, \bar{\theta})$ is a matched pair of Lie algebras, where
\begin{eqnarray*}
\bar{\phi}(x)(B(y), y)&=&(B([x, y]), [x, y]), \\
\bar{\theta}(B(y), y)x&=&B([x, y])+[B(y), x], \quad \forall x\in\g, (B(y), y)\in Gr(B).
\end{eqnarray*}
Denote by $E$ the connected and simply connected Lie group integrating $Gr(B)$. By Example \ref{exint}, it shows that $(\g, B)$ can not be integrated into a Rota-Baxter Lie group. Thus by Theorem \ref{intthm}, it follows that the matched pair of Lie algebras $(\g, Gr(B); \bar{\phi}, \bar{\theta})$ can not be integrated into a matched pair of $G$ and $E$.

\vspace{2mm}
\noindent
{\bf Declaration of interests. } The authors have no conflict of interest to declare that are relevant to this article.

\noindent
{\bf Data availability. } Data sharing is not applicable to this article as no new data were created or analyzed in this study.

 \end{document}